\documentclass[11pt]{article}

\usepackage{amsmath,amssymb,amsthm,amscd,amsfonts}
\usepackage{epsfig,pinlabel,paralist}
\usepackage{mathtools}
\usepackage[vmargin=1in, hmargin=1.25in]{geometry}
\usepackage[font=small,format=plain,labelfont=bf,up,textfont=it,up]{caption}
\usepackage{calc}
\usepackage{etoolbox}
\usepackage{microtype}
\usepackage[all,cmtip]{xy}

\usepackage[bookmarks, bookmarksdepth=2, colorlinks=true, linkcolor=blue, citecolor=blue, urlcolor=blue]{hyperref}

\apptocmd{\thebibliography}{\raggedright}{}{}

\numberwithin{equation}{section}

\theoremstyle{plain}
\newtheorem{theorem}{Theorem}[section]
\newtheorem{maintheorem}{Theorem}

\newtheorem{proposition}[theorem]{Proposition}
\newtheorem{lemma}[theorem]{Lemma}

\newtheorem*{claim}{Claim}

\newtheorem{casea}{Case}

\theoremstyle{definition}

\newtheorem{defn}[theorem]{Definition}
\newenvironment{definition}[1][]{\begin{defn}[#1]\pushQED{\qed}}{\popQED \end{defn}}

\theoremstyle{remark}
\newtheorem{rmk}[theorem]{Remark}
\newenvironment{remark}[1][]{\begin{rmk}[#1] \pushQED{}}{\popQED \end{rmk}}

\newtheorem{eg}[theorem]{Example}
\newenvironment{example}[1][]{\begin{eg}[#1] \pushQED{\qed}}{\popQED \end{eg}}

\DeclareMathOperator{\Hom}{Hom}

\DeclareMathOperator{\Ker}{ker}

\DeclareMathOperator{\GL}{GL}
\DeclareMathOperator{\SL}{SL}

\DeclareMathOperator{\Mat}{Mat}

\newcommand\R{\ensuremath{\mathbb{R}}}
\newcommand\C{\ensuremath{\mathbb{C}}}
\newcommand\Z{\ensuremath{\mathbb{Z}}}
\newcommand\Q{\ensuremath{\mathbb{Q}}}

\newcommand\Field{\ensuremath{\mathbb{F}}}
\newcommand\bbG{\ensuremath{\mathbb{G}}}

\DeclareMathOperator{\HH}{H}
\DeclareMathOperator{\CC}{C}

\newcommand\RH{\ensuremath{\widetilde{\HH}}}

\DeclareMathOperator{\Ind}{Ind}

\DeclareMathOperator{\link}{lk}
\DeclareMathOperator{\rk}{rk}
\DeclareMathOperator{\Interior}{Int}

\newcommand\Set[2]{\ensuremath{\left\{\text{#1 $|$
        #2}\right\}}}
\newcommand\SetLong[2]{\ensuremath{\left \{ #1
      \mathrel{}\middle|\mathrel{} #2 \right \} }}

\newcommand\cO{\ensuremath{\mathcal{O}}}

\newcommand\cF{\ensuremath{\mathcal{F}}}
\newcommand\cE{\ensuremath{\mathcal{E}}}

\newcommand\fD{\ensuremath{\mathfrak{D}}}
\newcommand\fF{\ensuremath{\mathfrak{F}}}
\newcommand\fP{\ensuremath{\mathfrak{P}}}

\newcommand\tC{\ensuremath{\widetilde{C}}}
\newcommand\tsigma{\ensuremath{\widetilde{\sigma}}}

\newcommand\oC{\ensuremath{\overline{C}}}
\newcommand\omQ{\ensuremath{\overline{\Q}}}

\newcommand\hB{\ensuremath{\widehat{B}}}

\DeclareMathOperator{\St}{St}
\DeclareMathOperator{\vcd}{vcd}
\DeclareMathOperator{\cl}{cl}
\DeclareMathOperator{\CL}{CL}
\newcommand\Tits{\ensuremath{\mathcal{T}}}
\newcommand\Lines{\ensuremath{\mathcal{L}}}
\newcommand\bfG{\ensuremath{\mathbf{G}}}

\newcommand\bfS{\ensuremath{\mathbf{S}}}
\newcommand\Id{\ensuremath{\operatorname{Id}}}
\newcommand\cptK{\ensuremath{\mathfrak{K}}}

\providecommand{\abs}[1]{\left\lvert#1\right\rvert}

\newcommand{\p}[1]{{\bf #1.}}

\title{\vspace{-40pt}The dualizing module and top-dimensional cohomology group of $\GL_n(\cO)$\vspace{-15pt}}
\author{Andrew Putman\thanks{Supported in part by NSF grant DMS-1811210} \and Daniel Studenmund\thanks{Supported in part by NSF grant DMS-1547292}}
\date{}

\begin{document}

\vspace{-10pt}
\maketitle

\vspace{-18pt}
\begin{abstract}
\noindent
For a number ring $\cO$, Borel and Serre proved that $\SL_n(\cO)$ is a virtual duality
group whose dualizing module is the Steinberg module.  They also proved that $\GL_n(\cO)$ is
a virtual duality group.  In contrast to $\SL_n(\cO)$, we prove that the dualizing
module of $\GL_n(\cO)$ is sometimes the Steinberg module, but sometimes instead is
a variant that takes into account a sort of orientation.
Using this, we obtain vanishing and nonvanishing theorems for the cohomology of 
$\GL_n(\cO)$ in its virtual cohomological dimension.  
\end{abstract}

\setlength{\parskip}{0pt}
\tableofcontents
\setlength{\parskip}{\baselineskip}

\section{Introduction}
\label{section:introduction}

The following contrasting theorems are two of the main results of this paper.
Let $\cl(\cO)$ denote the class group of a number ring $\cO$.

\begin{maintheorem}[Vanishing]
\label{maintheorem:highsmall}
Let $\cO$ be the ring of integers in a number field $K$ and let $\vcd$ be the
virtual cohomological dimension of $\GL_n(\cO)$.  Assume that $n$ is even
and that $\cO^{\times}$ contains an element of norm $-1$.  Also, letting
$r$ and $2s$ be the number of real and complex embeddings of $K$, assume
that $r+s \geq n$.
Then $\HH^{\vcd}(\GL_n(\cO);\Q) = 0$.
\end{maintheorem}

\begin{maintheorem}[Nonvanishing]
\label{maintheorem:highbig}
Let $\cO$ be the ring of integers in a number field $K$ and let $\vcd$ be the
virtual cohomological dimension of $\GL_n(\cO)$.  Assume either that $n$
is odd or that $\cO^{\times}$ does not contain an element of norm $-1$.  Then
the dimension of $\HH^{\vcd}(\GL_n(\cO);\Q)$ is at least
$(|\cl(\cO)|-1)^{n-1}$.
\end{maintheorem}

In the rest of the introduction, we will explain the origin and motivation for
these results.  In particular, we will explain why the parity of $n$ and
the (non)existence of elements in $\cO^{\times}$ of norm $-1$ should have
something to do with the cohomology of $\GL_n(\cO)$ in its virtual cohomological
dimension.

\begin{remark}
In light of the dichotomy suggested by Theorems \ref{maintheorem:highsmall} and \ref{maintheorem:highbig}, it
is natural to wonder which of their hypotheses are necessary.  In particular, it is unclear
whether the restrictive hypothesis $r+s \geq n$ is needed in Theorem \ref{maintheorem:highsmall}.  We
will discuss this at the end of the introduction.
\end{remark}

\begin{remark}
Theorem \ref{maintheorem:highbig} is closely connected to a recent theorem
of Church--Farb--Putman \cite{ChurchFarbPutman} that says that if $\nu$
is the virtual cohomological dimension of $\SL_n(\cO)$, then
the dimension of $\HH^{\nu}(\SL_n(\cO);\Q)$ is at least
$(|\cl(\cO)|-1)^{n-1}$.  Note that no assumption on $n$ or $\cO^{\times}$
is necessary.  The paper \cite{ChurchFarbPutman} also proves a vanishing
theorem for $\HH^{\nu}(\SL_n(\cO);\Q)$ that bears a superficial relationship
to Theorem \ref{maintheorem:highsmall}, but in fact the mechanisms behind the results
are completely different.  We will discuss this more later in the introduction.
\end{remark}

\p{Duality} Let $\cO$ be the ring of integers in a number field $K$ and let $r$ and $2s$ be the
numbers of real and complex embeddings of $K$. A fundamental result of
Borel--Serre \cite{BorelSerreCorners} says that the virtual cohomological 
dimension of $\GL_n(\cO)$ is
\[\vcd = r\binom{n+1}{2}+s n^2 - n.\]
Even better, they proved that $\GL_n(\cO)$ is a virtual duality group of dimension $\vcd$.  By definition,
this means that there is a $\Z[\GL_n(\cO)]$-module $\fD$ called the {\em virtual dualizing module}
such that the following holds.  Let $G \subset \GL_n(\cO)$ be a finite-index subgroup, including
possibly $G = \GL_n(\cO)$.  Let $R$ be a commutative ring such that for all finite subgroups $F < G$, the
order $|F|$ is invertible in $R$.  We thus can take $R = \Z$ if $G$ is torsion-free and
$R = \Q$ in all cases.  Then for all $R[G]$-modules $M$, we have
\[\HH^{\vcd-i}(G;M) \cong \HH_i(G;M \otimes \fD)\]
for all $i \geq 0$.

\begin{remark}
In most treatments of virtual duality groups, the duality relation is only discussed for torsion-free
subgroups of finite index.  It is well-known that the above holds for subgroups with torsion, but
we do not know a source that gives a detailed proof of this.
We will describe how this works for $\GL_n(\cO)$ in \S \ref{section:reducedual}.
\end{remark}

Specializing to $i=0$ and $G = \GL_n(\cO)$ and $M = \Q$, this says that
\[\HH^{\vcd}(\GL_n(\cO);\Q) \cong \HH_0(\GL_n(\cO);\Q \otimes \fD) \cong (\Q \otimes \fD)_{\GL_n(\cO)},\]
where the subscript indicates that we are taking coinvariants.  Theorems 
\ref{maintheorem:highsmall} and \ref{maintheorem:highbig} can thus be translated
into results about the action of $\GL_n(\cO)$ on its virtual dualizing module $\fD$.
The third main result of this paper identifies $\fD$.

\p{Special linear group and and the Steinberg module}
To motivate this identification, we first explain the better-understood
case of $\SL_n(\cO)$.  Just like for $\GL_n(\cO)$, Borel--Serre proved 
that $\SL_n(\cO)$ is a virtual duality group of virtual cohomological dimension
\[\nu = r\binom{n+1}{2}+s n^2 - n - r - s + 1.\]
They also gave the following beautiful description of the virtual dualizing
module for $\SL_n(\cO)$: it is the {\em Steinberg module} for $\SL_n(K)$, which 
we now describe.  Let $\Tits_n(K)$ be the Tits building for $\SL_n(K)$, i.e.\ the
geometric realization of the poset of $K$-parabolic subgroups of $\SL_n$.  The
$K$-parabolic subgroups of $\SL_n$ are precisely the stabilizers of flags
\begin{equation}
\label{eqn:kparabolic}
0 \subsetneq V_0 \subsetneq \cdots \subsetneq V_r \subsetneq K^n,
\end{equation}
and $\Tits_n(K)$ can alternately be described as the simplicial complex
whose $r$-simplices are flags as in \eqref{eqn:kparabolic}.  The
Solomon--Tits theorem \cite{Solomon, BrownBuildings} says that $\Tits_n(K)$
is homotopy equivalent to a wedge of $(n-2)$-spheres.  The Steinberg module
$\St_n(K)$ is $\RH_{n-2}(\Tits_n(K))$.  The action of $\SL_n(\cO)$ on $\St_n(K)$
is the restriction to $\SL_n(\cO)$ of the one induced by the action of
$\SL_n(K)$ on $\Tits_n(K)$.

Borel--Serre proved their theorem by constructing a bordification of
the symmetric space for $\SL_n(\cO)$.  The boundary of this bordification
has a stratification whose combinatorics are encoded by those
of the $K$-parabolic subgroups of $\SL_n$.  As a result,
the boundary is homotopy equivalent to $\Tits_n(K)$.

\p{General linear group}
To prove that $\GL_n(\cO)$ is a virtual duality group, Borel--Serre 
constructed a bordification of its associated symmetric space
in terms of the $K$-parabolic subgroups of $\GL_n$.  Since the $K$-parabolic
subgroups of $\GL_n$ are also the stabilizers of flags in $K^n$, it
follows that the boundary of their bordification for $\GL_n(\cO)$ is
homotopy equivalent to $\Tits_n(K)$.  This might lead the reader to expect that the virtual dualizing
module for $\GL_n(\cO)$ is also the Steinberg module $\St_n(K)$.

Unfortunately, this is false (see, e.g., \cite[\S 3]{Reeder} and \cite[\S 3.1]{OldThmC}).  Here is an easy
example of this failure.  We would like
to thank Jeremy Miller and Peter Patzt for pointing it out to us.

\begin{example}
\label{example:level2}
The virtual cohomological dimension of $\GL_2(\Z)$ is $1$.
Let $\Gamma_2(2)$ denote the level-$2$ principal congruence subgroup of $\GL_2(\Z)$, i.e.\ the kernel
of the map $\GL_2(\Z) \rightarrow \GL_2(\Field_2)$ that reduces matrix
entries modulo $2$.  Letting $\fD$ be the virtual dualizing module for $\GL_2(\Z)$ and
thus also for its finite-index subgroup $\Gamma_2(2)$, we have
\[\HH^1(\Gamma_2(2);\Q) \cong \HH_0(\Gamma_2(2);\Q \otimes \fD) = (\Q \otimes \fD)_{\Gamma_2(2)},\]
where the subscripts indicate that we are taking the coinvariants.  As
the following calculations show, $\HH^1(\Gamma_2(2);\Q) = 0$ and
$(\Q \otimes \St_2(\Q))_{\Gamma_2(2)} \neq 0$, so $\St_2(\Q) \neq \fD$.
\setlength{\parskip}{0pt}
\begin{compactitem}
\item The group $\Gamma_2(2)$ is generated by the matrices
\[a = \left(\begin{matrix} 1 & 2 \\ 0 & 1 \end{matrix}\right)
\quad \text{and} \quad
b = \left(\begin{matrix} 1 & 0 \\ 2 & 1 \end{matrix}\right)
\quad \text{and} \quad
c = \left(\begin{matrix} -1 & 0 \\ 0 & 1 \end{matrix}\right)
\quad \text{and} \quad
d = \left(\begin{matrix} 1 & 0 \\ 0 & -1 \end{matrix}\right).\]
We have $c^2=d^2=1$.  Also, $c a c^{-1} = a^{-1}$ and $c b c^{-1} = b^{-1}$.  It
follows that all the generators become torsion in the abelianization of $\Gamma_2(2)$,
so $\HH^1(\Gamma_2(2);\Q) = 0$.
\item The space $\Tits_2(\Q)$ is the discrete set of lines in $\Q^2$.  Such lines
are in bijection with rank-$1$ direct summands of $\Z^2$, and thus can be reduced
modulo $2$ to give lines in $\Field_2^2$.  This gives a surjection
$\Tits_2(\Q) \twoheadrightarrow \Tits_2(\Field_2)$ and hence a surjection
$\pi\colon \St_2(\Q) \twoheadrightarrow \St_2(\Field_2)$.  Since $\pi$ is $\Gamma_2(2)$-invariant,
it induces a surjection 
\[(\Q \otimes \St_2(\Q))_{\Gamma_2(2)} \twoheadrightarrow \Q \otimes \St_2(\Field_2) \neq 0.\qedhere\]
\end{compactitem}
\end{example}
\setlength{\parskip}{\baselineskip}

What is happening in the above example is that $\GL_2(\Z)$ acts in an orientation-reversing
way on its symmetric space.  The identification of the Steinberg module for $\SL_n(\cO)$ passes
through Poincar\'{e}--Lefschetz duality, so to do the same for $\GL_n(\cO)$ we must
take into account orientations.  

\p{Dualizing module}
If $G$ is a group and $A$ is an abelian group and $\chi\colon G \rightarrow \{\pm 1\}$ 
is a homomorphism, then let $A_{\chi}$ denote $A$ endowed with the $\Z[G]$-module
structure arising from the action 
\[g \cdot a = \chi(g) \cdot a \quad \text{for all $g \in G$ and $a \in A$}.\]
Our third main theorem is then the following.  Recall that the group of units 
$\cO^{\times}$ is precisely the set of elements of $\cO$ whose norm is $\pm 1$.

\begin{maintheorem}[Dualizing module]
\label{maintheorem:dualizing}
Let $\cO$ be the ring of integers in a number field $K$ and let $\fD$ be the
virtual dualizing module of $\GL_n(\cO)$.  Letting
$\chi\colon \GL_n(\cO) \rightarrow \{\pm 1\}$ be the composition of
the determinant homomorphism with the norm map $\cO^{\times} \rightarrow \{\pm 1\}$,
we then have $\fD \cong \St_n(K) \otimes (\Z_{\chi})^{\otimes(n-1)}$.
\end{maintheorem}

The virtual dualizing module of $\GL_n(\cO)$ is thus different from $\St_n(K)$ if and only
if $n$ is even and $\cO^{\times}$ has an element of norm $-1$.  This latter condition
forces $\cO$ to have a real embedding, so for instance never holds for rings
of integers in imaginary quadratic fields.  Beyond this, it is poorly understood
which number rings have elements of norm $-1$, even for rings of integers in
real quadratic fields.

\begin{remark}
Theorem \ref{maintheorem:dualizing} seems to have been known to the experts, and results
like it are mentioned in the literature in several places (see, e.g., \cite[\S 3]{Reeder} and \cite[\S 3.1]{OldThmC}).
However, no source we are aware of contains a proof of it in complete generality.  Since
we need Theorem \ref{maintheorem:dualizing} for Theorems \ref{maintheorem:highsmall} and \ref{maintheorem:highbig},
we have taken this opportunity to fill this hole in the literature. 
\end{remark}

\p{Cohomology in the vcd}
Having identified the virtual dualizing module $\fD$ for $\GL_n(\cO)$ in Theorem 
\ref{maintheorem:dualizing}, we now discuss Theorems \ref{maintheorem:highsmall}
and \ref{maintheorem:highbig}, which concern
\[\HH^{\vcd}(\GL_n(\cO);\Q) \cong (\Q \otimes \fD)_{\GL_n(\cO)}.\]
The restriction of the $\GL_n(\cO)$-module $\fD$ to $\SL_n(\cO)$ is
simply the Steinberg module $\St_n(K)$.  Letting $\nu$ be the virtual cohomological
dimension of $\SL_n(\cO)$, we thus have
\[\HH^{\nu}(\SL_n(\cO);\Q) \cong (\Q \otimes \St_n(K))_{\SL_n(\cO)} = (\Q \otimes \fD)_{\SL_n(\cO)}.\]
In \cite{ChurchFarbPutman}, Church--Farb--Putman proved two results about these
$\SL_n(\cO)$-coinvariants.

The first result of \cite{ChurchFarbPutman} generalizes
a theorem of Lee--Szczarba \cite[Theorem 4.1]{LeeSzczarba} that
says that if $\cO$ is Euclidean, then $(\Q \otimes \St_n(K))_{\SL_n(\cO)} = 0$.
The paper \cite{ChurchFarbPutman} says that this also holds if $\cl(\cO) = 0$
and $\cO$ has a real embedding.
Since $(\Q \otimes \fD)_{\GL_n(\cO)}$ is a quotient of 
$(\Q \otimes \St_n(K))_{\SL_n(\cO)}$, this implies that
under these assumptions we have
\[\HH^{\vcd}(\GL_n(\cO);\Q) \cong (\Q \otimes \fD)_{\GL_n(\cO)} = 0.\]
This vanishing result was already noted by Church--Farb--Putman; we
will later comment on its relationship to Theorem \ref{maintheorem:highsmall} (see
the ``Trouble'' paragraph below).

The second result of \cite{ChurchFarbPutman} says that
the dimension of $(\Q \otimes \St_n(K))_{\SL_n(\cO)}$ is at least
$(|\cl(\cO)|-1)^{n-1}$.  In fact, the proof in \cite{ChurchFarbPutman} actually
proves that the dimension of $(\Q \otimes \St_n(K))_{\GL_n(\cO)}$ is at least
$(|\cl(\cO)|-1)^{n-1}$, which is a stronger result.  The hypotheses of
Theorem \ref{maintheorem:highbig} are precisely those needed to ensure
that $\fD = \St_n(K)$, so Theorem \ref{maintheorem:highbig} immediately
follows.

\p{A tempting but wrong proof}
As we discussed above, Theorem \ref{maintheorem:highbig} follows from
Theorem \ref{maintheorem:dualizing} together with the work of Church--Farb--Putman,
so it only remains to discuss Theorem \ref{maintheorem:highsmall}.
In light of Theorem \ref{maintheorem:dualizing}, Theorem \ref{maintheorem:highsmall} is
equivalent to the assertion that under its assumptions, we have
\[(\St_n(K) \otimes \Q_{\chi})_{\GL_n(\cO)} = 0,\]
where $\chi\colon \GL_n(\cO) \rightarrow \{\pm 1\}$ is the composition of
the determinant homomorphism and the norm map $\cO^{\times} \rightarrow \{\pm 1\}$.
The Solomon--Tits theorem says that $\St_n(K)$ is generated by apartment classes
(see below for the definition),
and it is tempting to try to prove this by showing
that the images of these apartment classes in $\St_n(K) \otimes \Q_{\chi}$ vanish
in the $\GL_n(\cO)$-coinvariants.

The apartment classes $[A_B]$ are indexed by expressions $B = (L_1,\ldots,L_n)$
such that the $L_i$ are $1$-dimensional subspaces in $K^n$ with
$K^n = L_1 \oplus \cdots \oplus L_n$.  For such a $B$, let
$A_B$ denote the full subcomplex of $\Tits_n(K)$ spanned by the vertices
$\langle\text{$L_i$ $|$ $i \in I$}\rangle$, where $I \subset \{1,\ldots,n\}$
is a nonempty proper subset.  The complex $A_B$ is thus isomorphic to
the barycentric subdivision of the boundary of an $(n-1)$-simplex, and hence
is homeomorphic to an $(n-2)$-sphere.  The apartment class
is then the image
\[[A_B] \in \RH_{n-2}(\Tits_n(K)) = \St_n(K)\]
of the fundamental class of $A_B$.

The most straightforward way to show 
that $[A_B] \otimes 1 \in \St_n(K) \otimes \Q_{\chi}$ vanishes in
the $\GL_n(\cO)$-coinvariants would be to find some $g \in \GL_n(\cO)$ such
that $g([A_B]) = [A_B]$ but $\chi(g) = -1$; in the $\GL_n(\cO)$-coinvariants 
the elements $[A_B] \otimes 1$ and $g([A_B] \otimes 1) = -([A_B] \otimes 1)$
would then be equal.  For a general $B$, this seems
difficult.

However, it is easy to find such $g\in \GL_n(\cO)$ for the {\em integral apartments}, i.e.\ the
$B = (L_1,\ldots,L_n)$ such that 
\[\cO^n = (\cO^n \cap L_1) \oplus \cdots \oplus (\cO^n \cap L_n).\]
Indeed, for such a $B$ we can use a $g \in \GL_n(\cO)$ that scales
$L_1$ by an element of $\cO^{\times}$ whose norm is $-1$ and fixes
all the other $L_i$.  To prove Theorem \ref{maintheorem:highsmall}, it
would thus be enough to prove that $\St_n(K)$ is generated by integral apartments.

Ash--Rudolph \cite{AshRudolph} proved that if $\cO$ is Euclidean, then
$\St_n(K)$ is generated by integral apartments.  This was extended
in \cite{ChurchFarbPutman} to also include $\cO$ with a real
embedding and $\cl(\cO) = 0$.
Using a variant of the argument described above that avoids use of
the $\chi$-factor, \cite{ChurchFarbPutman} used this to prove their aforementioned vanishing
theorem.

\p{Trouble}
This leaves the cases of Theorem \ref{maintheorem:highsmall} that are
not consequences of Church--Farb--Putman's work, i.e.\ those where
$\cl(\cO) \neq 0$.  Unfortunately, \cite{ChurchFarbPutman} also proves
that if $\cl(\cO) \neq 0$, then $\St_n(K)$ is {\em not} generated by
integral apartments.  Finding a nice generating set for $\St_n(K)$ 
when $\cO$ is not Euclidean or a PID with a real embedding seems like a difficult problem, so we cannot
use one to prove Theorem \ref{maintheorem:highsmall}.

\p{What we do}
Our proof of Theorem \ref{maintheorem:highsmall} is thus by necessity entirely
different from the above sketch.  Recall that we are trying to prove that
$\HH_0(\GL_n(\cO);\St_n(K) \otimes \Q_{\chi})=0$.  Our proof of this has two
steps.  The first is to carefully study the action of $\GL_n(\cO)$ on the
simplicial chain complex of the Tits building to translate our theorem into
a sequence of results about the stable {\em untwisted} cohomology of $\GL_n(\cO)$.
The precise results we need are a bit technical, but the following special
case of one of them gives the general flavor.  Define $\CL_n(\cO)$ to
be the kernel of the homomorphism $\GL_n(\cO) \rightarrow \{\pm 1\}$ obtained
by composing the determinant and norm maps.
\begin{compactitem}
\item Let $r$ and $2s$ be the numbers of real and complex embeddings of $K$.
Then the action of $\GL_n(\cO)$ on its normal subgroup $\CL_n(\cO)$ induces
the trivial action on $\HH_k(\CL_n(\cO);\Q)$ for
$0 \leq k \leq \min(n,r+s)-1$.
\end{compactitem}
See Proposition \ref{proposition:projectiveactiontrivial} for a more general
statement.  For some range of $k$ (up to around $\frac{n}{2}$), this could be
easily deduced from Borel's \cite{BorelStable} computation of the stable rational
cohomology of $\SL_n(\cO)$.  However, we really need the whole range of values
of $k$ above -- even getting a result that was off by $1$ would cause everything
to break!  

There is a vast literature on homological stability results, and
we use some of the technology developed there in a rather non-standard way to
prove our theorem.  It is a bit surprising that while an optimal homological stability
theorem for $\SL_n(\cO)$ is not known, the technology that has been developed is
just barely strong enough to prove a result like the above that gives information
well outside the known stable range.

\p{Necessity of hypotheses}
Theorem \ref{maintheorem:highsmall} has three hypotheses:
\setlength{\parskip}{0pt}
\begin{compactitem}
\item $n$ is even, and
\item $\cO^{\times}$ contains an element of norm $-1$, and
\item $r+s \geq n$, where $r$ and $2s$ are the numbers of real and complex embeddings of $K$.
\end{compactitem}
The third of these is quite restrictive, and it is natural to wonder whether or not it is necessary.
For $\cO^{\times}$ to contain an element of norm $-1$, it is necessary for $\cO$ to have a real
embedding.  Lee and Szczarba \cite[Theorem 4.1]{LeeSzczarba} also 
proved\footnote{The reference \cite[Theorem 4.1]{LeeSzczarba} actually states
that if $\cO$ is Euclidean, then $\HH_0(\SL_n(\cO);\St_n(K)) = 0$ for $n \geq 2$.  Bieri--Eckmann duality
implies that this is equivalent to the vanishing of the rational cohomology of $\SL_n(\cO)$ in its virtual
cohomological dimension.  Using the Hochschild--Serre spectral sequence associated to the short
exact sequence $1 \rightarrow \SL_n(\cO) \rightarrow \GL_n(\cO) \rightarrow \cO^{\times} \rightarrow 1$,
this implies that $\HH^{\vcd}(\GL_n(\cO);\Q) = 0$.}
that if $\cO$ is Euclidean, then
$\HH^{\vcd}(\GL_n(\cO);\Q) = 0$ for $n \geq 2$.  The simplest number rings not covered by Lee--Szczarba's theorem for
which an element of norm $-1$ might exist are therefore real quadratic $\cO$ with positive class
numbers.  For these, the group $\GL_2(\cO)$ is covered by Theorem \ref{maintheorem:highsmall} (which says
that vanishing does not hold), while
the group $\GL_3(\cO)$ is covered by Theorem \ref{maintheorem:highbig}.  The smallest possible
interesting examples not covered by these known results are thus $\GL_4(\cO)$ for real quadratic
number rings $\cO$ with positive class number and elements of norm $-1$.  Unfortunately, these
are complicated enough that we are unaware of any computational data concerning them.\setlength{\parskip}{\baselineskip}

\p{Outline}
The two theorems above we must prove are Theorems \ref{maintheorem:dualizing}
and \ref{maintheorem:highsmall}.  We prove Theorem \ref{maintheorem:dualizing}
in \S \ref{section:identifydual}, and we start the proof of Theorem 
\ref{maintheorem:highsmall} in \S \ref{section:reductioni}, which reduces
it to results proved in subsequent sections.

\p{Acknowledgments}
The first author would like to thank Thomas Church and Benson Farb for many inspiring
conversations about the Steinberg module.  The second author would
like to thank Dave Witte Morris and Kevin Wortman for help
understanding the construction of Borel--Serre. We both would like to thank Jeremy Miller
and Peter Patzt for showing us Example \ref{example:level2} and asking us what
was going on with the virtual dualizing module for $\GL_n(\Z)$.  We also would like to thank
Will Sawin for showing us how to prove Lemma \ref{lemma:identifyclosure} below.  Finally,
we would like to thank Khalid Bou-Rabee for some helpful comments.

\section{Identifying the virtual dualizing module}
\label{section:identifydual}

In this section, we prove Theorem \ref{maintheorem:dualizing}.  There
are two subsections.  In \S \ref{section:reducedual}, we use
standard techniques to reduce ourselves to the existence of an action
of $\GL_n(\cO)$ on a space with appropriate properties.  This space
was constructed by Borel--Serre \cite{BorelSerreCorners}, but they did
not verify one key property we need. In \S
\ref{section:dualconstruction} we recall the construction of the space
and verify the key property.

\subsection{Reduction to a group action}
\label{section:reducedual}

In this section, we will show how Theorem \ref{maintheorem:dualizing}
follows from the following proposition, which is essentially due to Borel--Serre \cite{BorelSerreCorners}.
However, they did not verify all the properties in it, in particular conclusion (iv).

\begin{proposition}
  \label{proposition:spaceaction}
Let $\cO$ be the ring of integers in a number field $K$, and
let $r$ and $2s$ be the numbers of real and complex embeddings of $K$.  Let $\chi\colon \GL_n(\cO) \rightarrow \{\pm 1\}$
be the composition of the determinant with the norm map $\cO^{\times} \rightarrow \{\pm 1\}$.  
Then there exists a smooth contractible manifold with corners $\overline X$ such that the following hold.
\setlength{\parskip}{0pt}
\begin{compactenum}[(i)]
    \item The group $\GL_n(\cO)$ acts smoothly, properly discontinuously, and cocompactly on $\overline{X}$.
    \item The boundary $\partial \overline{X}$ is homotopy equivalent to the Tits building $\Tits_n(K)$, and the restriction
of the $\GL_n(\cO)$-action to $\partial \overline{X}$ corresponds to the usual action of $\GL_n(\cO)$ on $\Tits_n(K)$.
    \item The dimension of $\overline{X}$ is $d = r \binom{n+1}{2}+s n^2 - 1$.
    \item For $g\in \GL_n(\cO)$, the action of $g$ on $\overline{X}$ reverses orientation if and only
if $n$ is even and $\chi(g) = -1$.
\end{compactenum}
\setlength{\parskip}{\baselineskip}
\end{proposition}

We will explain how to extract Proposition \ref{proposition:spaceaction} from 
Borel--Serre's work in \S \ref{section:dualconstruction}.  Here we show how to use
it to derive Theorem \ref{maintheorem:dualizing}.  This derivation
is mostly standard, but we spell it out since we do not know a source that 
carefully deals with orientations and non-free actions.  Indeed, many sources talk
about virtual duality groups, but ignore the fact that they are also $\Q$-duality
groups if they have torsion, which is essential for our applications.

We need two lemmas.  The first is a tiny generalization of a familiar fact
about representations of groups over fields.  Recall that if $F$ is a group
and $\chi\colon F \rightarrow \{\pm 1\}$ is a homomorphism and $R$ is a commutative
ring, then $R_{\chi}$ is the $R[F]$-module whose underlying $R$-module is $R$
and where $g \in F$ acts as multiplication by $\chi(g)$.

\begin{lemma}
\label{lemma:split}
Let $F$ be a finite group, $\chi\colon F \rightarrow \{\pm 1\}$ be a homomorphism,
and $R$ be a commutative ring such that $|F|$ is invertible in $R$.  Then
the $R[F]$-module $R_{\chi}$ is projective.
\end{lemma}
\begin{proof}
The surjection $\pi\colon R[F] \rightarrow R_{\chi}$ defined via the formula
\[\pi(g) = \chi(g) \in R \quad \quad (g \in F)\]
splits via the homomorphism $\iota\colon R_{\chi} \rightarrow R[F]$ taking $1 \in R_{\chi}$ to
\[\frac{1}{|F|} \sum_{g \in F} \chi(g) \cdot g.\]
Thus $R_{\chi}$ is a direct summand of the free $R[F]$-module $R[F]$, and is therefore projective.
\end{proof}

\begin{lemma}
\label{lemma:projectiveresolution}
Let $G$ be a group and let $Z$ be a contractible simplicial complex upon which $G$ acts properly discontinuously and
cocompactly.  Let $R$ be a commutative ring such that for all finite subgroups $F<G$, the order
$|F|$ is invertible in $R$.  Then the simplicial chain complex $\CC_{\bullet}(Z;R)$ is a
resolution of $R$ by finitely generated projective $R[G]$-modules.
\end{lemma}
\begin{proof}[{Proof (compare to \cite[Lemma 3.2]{ChurchPutman})}.]
The fact that $Z$ is contractible implies that $\CC_{\bullet}(Z;R)$ is a resolution of $R$.  We must
prove that each $\CC_n(Z;R)$ is a finitely generated projective $R[G]$-module.  For an oriented
$n$-simplex $\sigma$ of $Z$, let $M_{\sigma}$ be the $R[G]$-submodule of $\CC_n(Z;R)$ generated
by the basis element corresponding to $\sigma$.
Since $G$ acts cocompactly on $Z$, there are finitely many orbits of the action of $G$ on the set of $n$-simplices
of $Z$.  Let $\{\sigma(1),\ldots,\sigma(m)\}$ be a set of orbit representatives for this action.  Fixing
an orientation on each $\sigma(i)$, we have
\[\CC_n(Z;R) = \bigoplus_{i=1}^m M_{\sigma(i)}.\]
It is thus enough to prove that each $M_{\sigma(i)}$ is a projective $R[G]$-module.  Let $G_{\sigma(i)}$ be
the setwise stabilizer of $\sigma$.  Since the action of $G$ on $Z$ is properly discontinuous, $G_{\sigma(i)}$
is a finite subgroup of $G$.  The action of $G_{\sigma(i)}$ on $Z$ might reverse the orientation of $\sigma(i)$.
Let $\chi\colon G_{\sigma(i)} \rightarrow \{\pm 1\}$ be the homomorphism that records whether or not
an element of $G_{\sigma(i)}$ reverses the orientation of $\sigma(i)$.  We then have
\[M_{\sigma(i)} = \Ind_{G_{\sigma(i)}}^G R_{\chi}.\]
Lemma \ref{lemma:split} says that $R_{\chi}$ is a projective $R[G_{\sigma(i)}]$-module, i.e.\ a direct
summand of a free $R[G_{\sigma(i)}]$-module.  Since
\[\Ind_{G_{\sigma(i)}}^G R[G_{\sigma(i)}] \cong R[G],\]
it follows that $M_{\sigma(i)}$ is also a direct summand of a free $R[G]$-module, and is thus projective,
as desired.
\end{proof}

\begin{proof}[Proof of Theorem \ref{maintheorem:dualizing}, assuming
  Proposition \ref{proposition:spaceaction}]
Let us first recall what must be proved.  This requires introducing a large amount of notation:
\setlength{\parskip}{0pt}
\begin{compactitem}
\item Let $\cO$ be the ring of integers in a number field $K$.
\item Let $\chi\colon \GL_n(\cO) \rightarrow \{\pm 1\}$ be the
composition of the determinant homomorphism and the norm map $\cO^{\times} \rightarrow \{\pm 1\}$.
\item Let $r$ and $2s$ be the numbers of real and complex embeddings of $K$.
\item Let $\vcd = r\binom{n+1}{2}+s n^2 - n$.
\item Let $G$ be a finite-index subgroup of $\GL_n(\cO)$.
\item Let $R$ be a commutative ring such that for all finite subgroups $F<G$, the order $|F|$ is invertible in $R$.
\end{compactitem}
We must prove that $G$ is an $R$-duality group of dimension $\vcd$
with $R$-dualizing module $\St_n(K) \otimes (R_{\chi})^{\otimes(n-1)}$.  Since
this purported dualizing module is a free $R$-module, the standard theory
of $R$-duality group (see, e.g.\ \cite[\S 9]{BieriNotes}) says that this is
equivalent to showing that\setlength{\parskip}{\baselineskip}
\begin{equation}
\label{eqn:toidentify}
\HH^k(G;R[G]) \cong \begin{cases}
\St_n(K) \otimes (R_{\chi})^{\otimes(n-1)} & \text{if $k = \vcd$},\\
0 & \text{otherwise}
\end{cases}
\end{equation}
for all $k \geq 0$.

Let $X$ and $\overline{X}$ be as in Proposition \ref{proposition:spaceaction}, so $\overline{X}$
is a
\[d = r \binom{n+1}{2}+s n^2 - 1\]
dimensional manifold with boundary.
Fix a $\GL_n(\cO)$-equivariant triangulation of $\overline{X}$.  Lemma \ref{lemma:projectiveresolution} implies 
that the simplicial chain complex
$\CC_{\bullet}(\overline{X};R)$ is a resolution of $R$ by finitely-generated projective $R[G]$-modules.

The proof of \cite[Proposition VIII.7.5]{BrownCohomology} now shows that
\begin{equation}
\label{eqn:identify1}
\HH^{k}(G;R[G]) \cong \HH^{k}_c(\overline{X};R).
\end{equation}
Let $R_{\text{or}}$ be the orientation module for the action of $\GL_n(\cO)$ on $\overline{X}$, so
$R_{\text{or}} = R$ and elements of $\GL_n(\cO)$ act on $R_{\text{or}}$ by $\pm 1$ depending on
whether or not they reverse the orientation of $\overline{X}$.
Conclusion (iv) of Proposition \ref{proposition:spaceaction} implies that
\begin{equation}
\label{eqn:identify2}
R_{\text{or}} \cong (R_{\chi})^{\otimes(n-1)}.
\end{equation} 
Applying Poincar\'e-Lefschetz duality, we see that as a $G$-module, we have
\begin{equation}
\label{eqn:identify3}
\HH_c^k(\overline{X};R)
\cong
\HH_{d-k}(\overline{X},\partial \overline{X};R) \otimes R_{\text{or}}.
\end{equation}
Using the fact that $\overline{X}$ is contractible, the long exact sequence of a pair gives
\begin{equation}
\label{eqn:identify4}
\HH_{d-k}(\overline{X},\partial \overline{X};R) \otimes R_{\text{or}} 
\cong
\RH_{d-k-1}(\partial \overline{X};R) \otimes R_{\text{or}} \cong \RH_{d-k-1}(\Tits_n(K)) \otimes R_{\text{or}}.
\end{equation}
Since the Tits building $\Tits_n(K)$ is homotopy equivalent to a wedge of $(n-2)$-spheres and
\[d - \vcd - 1 = \left(r \binom{n+1}{2}+s n^2 - 1\right) - \left(r\binom{n+1}{2}+s n^2 - n\right) -1 = n-2,\]
we have
\begin{equation}
\label{eqn:identify5}
\RH_{d-k-1}(\Tits_n(K)) \cong \begin{cases}
\St_n(K) & \text{if $k = \vcd$},\\
0 & \text{otherwise}.
\end{cases}
\end{equation}
Combining \eqref{eqn:identify1}--\eqref{eqn:identify5}, we obtain \eqref{eqn:toidentify}.
\end{proof}

\subsection{The Borel--Serre bordification}
\label{section:dualconstruction}

Let $\cO$ be the ring of integers in an algebraic number field $K$.
A space $\overline{X}$ satisfying the conclusions of Proposition \ref{proposition:spaceaction} was
constructed by Borel--Serre \cite{BorelSerreCorners}, who proved that
it satisfies the first three conclusions of that proposition.  In this
section, we recall their construction and verify that it also satisfies
the fourth conclusion.

\p{Algebraic groups setup}
Let $\bfG = R_{K/\Q}(\GL_n)$ be the $\Q$-algebraic group obtained as
the restriction of scalars of the $K$-algebraic group $\GL_n$.
We thus have $\bfG(\Q) \cong \GL_n(K)$ and $\bfG(\Z) \cong \GL_n(\cO)$. 
Let $r$ and $2s$ be the numbers of real and complex embeddings of $K$ and
let
\[G = \bfG(\R) = \GL_n(K \otimes_{\Q} \R) \cong \prod_{i=1}^{r} \GL_n(\R) \times \prod_{j=1}^s \GL_n(\C).\]
The group $\GL_n(\cO)$ is thus a discrete subgroup of the real Lie group $G$.
Let
\[
  \cptK = \prod_{i=1}^{r} \operatorname{O}(n) \times \prod_{j=1}^s \operatorname{U}(n),
\]
so $\cptK$ is a maximal compact subgroup of $G$.

\p{Center}
Recall that a {\em bordification} of a smooth manifold $Y$ is a smooth manifold with
corners $\overline{Y}$ such that $\Interior(\overline{Y}) = Y$.  The space
$\overline{X}$ constructed by Borel--Serre is a bordification of an appropriate
symmetric space $X$.
If we were working with a semisimple group like $\SL_n$, then $X$ would simply be $G/\cptK$.
To deal with a reductive group like $\GL_n$, we will have to further quotient $G/\cptK$ by 
the following subgroup of the center.  The center of the $K$-algebraic group $\GL_n$ is the multiplicative group $\bbG_m$, so
$Z(\bfG) = R_{K/\Q}(\bbG_m)$.  Let $\bfS$ be the maximal $\Q$-split torus in $Z(\bfG)$, so
\[\bfS(\Q) = \Q^{\times} < K^{\times} = Z(\bfG)(\Q).\]
Set $S = \bfS(\R) < G$.  Letting $\Id$ be the $n \times n$ identity matrix, we have
\[S = \Set{$(a\Id,\dotsc, a\Id)$}{$a\in \R^{\times}$} \leq Z(G).\]

\p{Symmetric space}
Let $X$ be the smooth manifold $G/(\cptK \cdot S)$.  The space $X$ is a symmetric space
of noncompact type, and is thus contractible.  This can be seen in an elementary
way using the Gram--Schmidt orthogonalization process.  Its dimension is
\begin{align}
\label{eqn:calcdim}
\dim(X) &= r(\dim(\GL_n(\R))-\dim(\operatorname{O}(n))) + s(\dim(\GL_n(\C))-\dim(\operatorname{U}(n)))-1 \\
&= r(n^2-\frac{n(n-1)}{2})+s(2n^2-n^2)-1 = r \frac{n(n+1)}{2} + s n^2 -1. \notag
\end{align}
Since $\GL_n(\cO) \cap \cptK \cdot S$ is
finite, the smooth and properly discontinuous action of $\GL_n(\cO)$
on $G$ by left multiplication descends to a smooth and properly
discontinuous action of $\GL_n(\cO)$ on $X$.  

\p{Borel--Serre bordification}
Borel--Serre \cite{BorelSerreCorners}
prove that $X$ has the following properties.

\begin{theorem}[{Borel--Serre, \cite{BorelSerreCorners}}]
\label{theorem:borelserre}
Let the notation be as above.  The manifold $X$ has a bordification $\overline{X}$ with the following
properties.
\setlength{\parskip}{0pt}
\begin{compactenum}[(i)]
    \item\label{item:actionnice} The action of $\GL_n(\cO)$ on $X$ extends to a smooth, properly discontinuous, and cocompact
action on $\overline{X}$.
    \item\label{item:rightboundary} The boundary $\partial \overline{X}$ is homotopy equivalent to the Tits building $\Tits_n(K)$, and the restriction
of the $\GL_n(\cO)$-action to $\partial \overline{X}$ corresponds to the usual action of $\GL_n(\cO)$ on $\Tits_n(K)$.
\end{compactenum}
\setlength{\parskip}{\baselineskip}
\end{theorem}
\begin{proof}
The space $X$ is a ``space of type
$S-\Q$ for $G$'' in the language of Borel--Serre; see
\cite[2.5(2)]{BorelSerreCorners}. Borel--Serre construct a
manifold with corners $\overline{X}$ containing $X$ as an open
submanifold; see \cite[7.1]{BorelSerreCorners}. Their construction
satisfies (\ref{item:actionnice}) by \cite[9.3]{BorelSerreCorners} and
satisfies (\ref{item:rightboundary}) by
\cite[8.4.2]{BorelSerreCorners}. 
\end{proof}

\p{What remains}
Theorem \ref{theorem:borelserre} says that the bordification $\overline{X}$ satisfies
the first two conclusions of Proposition \ref{proposition:spaceaction}, and
\eqref{eqn:calcdim} shows that $\overline{X}$ satisfies the third conclusion.
To prove Proposition \ref{proposition:spaceaction}, we must therefore only
verify the fourth conclusion, which identifies the elements of $\GL_n(\cO)$ that
preserve the orientation of $\overline{X}$.  Since a diffeomorphism of a smooth manifold with
corners is orientation-preserving if and only if its restriction to the interior
is orientation-preserving, it is enough to determine which elements of $\GL_n(\cO)$
preserve the orientation of $X$.  The advantage of doing this is that the whole Lie group
$G$ acts on $X$, and in fact we will determine which elements of $G$ preserve the
orientation of $X$.

Define $\chi\colon G \rightarrow \R$ via the formula
\[\chi(g_1,\ldots,g_r,g'_1,\ldots,g'_{s}) = \det(g_1) \cdots \det(g_r) \cdot |\det(g'_{1})| \cdots |\det(g'_{s})|.\]
The restriction of $\chi$ to $\GL_n(K) \subset \GL_n(K)$ (and hence to $\GL_n(\cO)$) is
is the composition of the determinant homomorphism
$\GL_n(K) \rightarrow K^{\times}$ with the norm map $K^{\times} \rightarrow \Q^{\times}$.  From this,
we see that the following lemma generalizes the fourth conclusion of Proposition \ref{proposition:spaceaction}.
 
\begin{lemma}
\label{lemma:orientation}
Let the notation be as above.  For $g\in G$, the action of $g$ on $X$ reverses orientation if and only
if $n$ is even and $\chi(g) < 0$.
\end{lemma}

\noindent
Once we prove Lemma \ref{lemma:orientation}, the proof of Proposition \ref{proposition:spaceaction}
will be complete.  Before we do this, we must discuss two preliminary results.  

\p{Homogeneous spaces and orientations}
The first is
the following lemma.  To interpret it, observe that if $M$ is a connected
orientable manifold, then the question of whether a homeomorphism of $M$ preserves the orientation
is independent of a choice of orientation.

\begin{lemma}
\label{lemma:orientsetup}
Let $H$ be a Lie group and let $M$ be smooth connected orientable
homogeneous space for $H$.  Fix a basepoint $p \in M$.
Then the action of $H$ on $M$ preserves the orientation
of $M$ if and only if the stabilizer $H_p$ preserves the orientation
of the tangent space $T_p M$.
\end{lemma}
\begin{proof}
If the action of $H$ on $M$ preserves the orientation of $M$, then clearly
$H_p$ preserves the orientation of $T_p M$.  We must prove the converse.  Assume
that $H_p$ preserves the orientation of $T_p M$.
Since $M$ is connected, it is enough to construct an $H$-invariant orientation
of $M$.  For this, let $\omega$ be an orientation on $T_p M$.
We can then define an orientation on $M$ by letting the orientation on $T_q M$ for $q \in M$
be $h_{\ast}(\omega)$, where $h \in H$ satisfies $h(p) = q$.  This is
independent of the choice of $h$, and clearly gives a $H$-invariant
orientation on $M$.
\end{proof}

\p{Ignoring the center}
Our second lemma will allow us to ignore the difference between $X = G/(\cptK \cdot S)$ and $G/\cptK$:

\begin{lemma}
\label{lemma:ignorecenter}
Let the notation be as above.  For $g \in G$, the action of $g$ on $X$ preserves orientation
if and only if the action of $g$ on $G/\cptK$ preserves orientation.
\end{lemma}
\begin{proof}
Define $\Psi\colon G \rightarrow \R_{>0}$ via the formula
\[\Psi(g_1,\ldots,g_{r+s}) = \abs{\det g_1} \cdots \abs{\det g_{r+s}}.\]
Via $\Psi$, the group $G$ acts in an orientation-preserving way on $\R_{>0}$.  To prove the
lemma, it is thus enough to prove that there is a $G$-equivariant homeomorphism
\[\R_{>0} \times X \cong G/\cptK.\]
To do this, it is enough to prove that $\R_{>0} \times X$ is the homogeneous $G$-space $G/\cptK$.

Since $\Psi(S) = \R_{>0}$, the subgroup $S<G$ acts transitively on $\R_{>0}$.  The subgroup $S$
lies in the center of $G$, so $S$ acts trivially on $X = G/(\cptK \cdot S)$.  Together, these
facts imply that $\R_{>0} \times X$ is a homogeneous $G$-space.  As
a $G$-space, $\R_{>0}$ is isomorphic to $G/\Ker(\Psi)$.  We conclude that
$\R_{>0} \times X$ is isomorphic as a $G$-space to $G$ modulo
\[\Ker(\Psi) \cap (\cptK \cdot S) = \cptK \cdot (\Ker(\Psi) \cap S) = \cptK \cdot (S \cap \cptK) = \cptK.\]
Here we are using the fact that $\cptK \subset \Ker(\Psi)$ and that $S$ is central.  The lemma follows.
\end{proof}

\p{Completing the proof}
We finally prove Lemma \ref{lemma:orientation}, thus completing the proof of Proposition
\ref{proposition:spaceaction}.

\begin{proof}[Proof of Lemma \ref{lemma:orientation}]
By Lemma \ref{lemma:ignorecenter}, it is enough to prove that the action of $g \in G$ on $G/\cptK$
reverses orientation if and only if $n$ is even and $\chi(g) < 0$.  By definition,
\[G = \left(\prod_{i=1}^r \GL_n(\R)\right) \times
\left(\prod_{j=1}^s \GL_n(\C)\right)\]
and
\[G/\cptK = \left(\prod_{i=1}^r \frac{\GL_n(\R)}{\operatorname{O}(n)}\right) \times
\left(\prod_{j=1}^s \frac{\GL_n(\C)}{\operatorname{U}(n)}\right).\]
The action of $G$ on $G/\cptK$ respects these product decompositions.  It follows that for
$g = (g_1,\ldots,g_{r},g'_{1},\ldots,g'_{s}) \in G$,
the action of $g$ on $G/\cptK$ reverses orientation if and only if
\begin{align*}
\#&\Set{$1 \leq i \leq r$}{$g_i$ reverses orientation of $\frac{\GL_n(\R)}{\operatorname{O}(n)}$}\\
&+ \#
\Set{$1 \leq j \leq s$}{$g'_j$ reverses orientation of $\frac{\GL_n(\C)}{\operatorname{U}(n)}$}
\end{align*}
is odd.  Since $\GL_n(\C)$ is connected, the action of $g'_j$ will preserve orientation for all
$1 \leq j \leq s$.  What is more, since
\[\chi(g_1,\ldots,g_r,g'_1,\ldots,g'_{s}) = \det(g_1) \cdots \det(g_r) \cdot |\det(g'_{1})| \cdots |\det(g'_{s})|,\]
we see that $\chi(g) < 0$ if and only if
\[\#\Set{$1 \leq i \leq r$}{$\det(g_i)<0$}\]
is odd.  We conclude that to prove the lemma, it is enough to prove the following claim.

\begin{claim}
The subgroup of $\GL_n(\R)$ consisting of elements that fix the orientation 
of 
\[Y = \frac{\GL_n(\R)}{\operatorname{O}(n)}\]
is $\GL_n^{>0}(\R)$ if $n$ is even and is $\GL_n(\R)$ if $n$ is odd.
\end{claim}
Since $\GL_n(\R)$ has two components, the subgroup in question is either $\GL_n(\R)$ or 
$\GL_n^{>0}(\R)$, so it is enough to prove that $\GL_n(\R)$ itself preserves the orientation
on $Y$ if and only if $n$ is odd.

By Lemma \ref{lemma:orientsetup}, the group $\GL_n(\R)$ preserves the orientation on $Y$
if and only if $\operatorname{O}(n)$ preserves the orientation on the tangent space at the identity coset.
We can identify this tangent space as the quotient of Lie algebras 
\[V = \frac{\mathfrak{gl}_n(\R)}{\mathfrak{o}(n)},\]
and the action of $\operatorname{O}(n)$ on it is the one induced by conjugation.  

Since $\operatorname{O}(n)$ has only two components and the component of the identity clearly preserve the
orientation on this tangent space, it suffices to check a single element of the non-identity
component.  We will use the matrix $e_{11}(-1)$ obtained from the identity matrix by replacing
the entry at $(1,1)$ with $-1$.

For $1 \leq i,j \leq n$, let $a_{ij} \in \mathfrak{gl}_n(\R)$ be the matrix with a $1$ at position
$(i,j)$ and zeros elsewhere.  The vector space $V$ has a basis consisting of the cosets of
$\Set{$a_{ij}$}{$1 \leq i \leq j \leq n$}$.  Conjugation by $e_{11}(-1)$ fixes $a_{ii}$ for $1 \leq i \leq n$,
and also fixes $a_{ij}$ for $2 \leq i < j \leq n$.  However, conjugation by $e_{11}(-1)$ takes
$a_{1j}$ with $j \leq 2 \leq n$ to $-a_{1j}$.  This conjugation action thus negates precisely $(n-1)$ elements
of our basis, so the determinant of its action on $V$ is $(-1)^{n-1}$.  We conclude that 
$e_{11}(-1)$ preserves the orientation if and only if $n$ is odd, as desired.
\end{proof}

\section{Reduction I: the action on flag stabilizers is trivial}
\label{section:reductioni}

We now begin our proof of Theorem \ref{maintheorem:highsmall}.  In this
section, we reduce this theorem to proving that a certain action is trivial.

\p{Setup}
Let $\cO$ be the ring of integers in a number field $K$ such that $\cO$ has an 
element of norm $-1$.  Let $\chi\colon \GL_n(\cO) \rightarrow \{\pm 1\}$
be the composition of the determinant with the norm map
$\cO^{\times} \rightarrow \{\pm 1\}$, and define
$\CL_n(\cO) = \ker(\chi)$.  Let $\fF$ be a length-$q$ flag in $K^n$, i.e.\ an
increasing sequence of subspaces
\[0 \subsetneq \fF_0 \subsetneq \fF_1 \subsetneq \cdots \subsetneq \fF_q \subsetneq K^n.\]
By convention, the degenerate case $q=-1$ simply means the empty flag.  Define
$\GL_n(\cO,\fF)$ (resp.\ $\CL_n(\cO,\fF)$) to be the subgroup of $\GL_n(\cO)$ (resp.\ 
$\CL_n(\cO)$) that preserves $\fF$.  If $q=-1$, then $\GL_n(\cO,\fF) = \GL_n(\cO)$
and $\CL_n(\cO,\fF) = \CL_n(\cO)$.  The group $\CL_n(\cO,\fF)$ is a normal subgroup
of $\GL_n(\cO,\fF)$ of index at most $2$.  See Remark \ref{remark:achievedet} below
for a proof that it has index equal to $2$.

\p{The reduction}
The proof of the following proposition begins in
\S \ref{section:reductionii}.

\begin{proposition}
\label{proposition:flagactiontrivial}
Let $\cO$ be the ring of integers in a number field $K$ and let
$\fF$ be a flag in $K^n$.
Assume that $\cO^{\times}$ has an element of norm $-1$, and
let $r$ and $2s$ be the numbers of real and complex embeddings of $K$.
Then the action of $\GL_n(\cO,\fF)$ on its normal subgroup $\CL_n(\cO,\fF)$ induces
the trivial action on $\HH_k(\CL_n(\cO,\fF);\Q)$ for
$0 \leq k \leq \min(r+s,n)-1$.
\end{proposition}

Here we will assume the truth of Proposition \ref{proposition:flagactiontrivial}
and use it to prove Theorem \ref{maintheorem:highsmall}.

\begin{proof}[Proof of Theorem \ref{maintheorem:highsmall}, assuming Proposition \ref{proposition:flagactiontrivial}]
We start by recalling what we must prove.  Let $\cO$ be the ring of integers in a
number field $K$ and let $\vcd$ be the
virtual cohomological dimension of $\GL_n(\cO)$.  Assume that the following hold.
\setlength{\parskip}{0pt}
\begin{compactitem}
\item $n$ is even.
\item $\cO^{\times}$ contains an element of norm $-1$.
\item Letting $r$ and $2s$ be the numbers of real and complex embeddings of $K$, we
have $r+s \geq n$.
\end{compactitem}
Our goal is then to prove that $\HH^{\vcd}(\GL_n(\cO);\Q) = 0$.  Let \setlength{\parskip}{\baselineskip}
$\chi\colon \GL_n(\cO) \rightarrow \{\pm 1\}$ be the composition of
the determinant with the norm map $\cO^{\times} \rightarrow \{\pm 1\}$.
Applying Borel--Serre duality and Theorem \ref{maintheorem:dualizing}, we see
that our goal is equivalent to showing that
\[\HH_0\left(\GL_n\left(\cO\right); \St_n\left(K\right) \otimes \Q_{\chi}\right) = 0.\]
For this, we must study the action of $\GL_n(\cO)$ on the chain complex for the building
$\Tits_n(K)$.

Let $\tC_{\bullet}$ be the usual augmented chain complex calculating the reduced simplicial
homology of $\Tits_n(K)$, so $\tC_{-1} = \Z$ and
\[\HH_k\left(\tC_{\bullet}\right) = \begin{cases}
\St_n\left(K\right) & \text{if $k=n-2$},\\
0 & \text{if $k \neq n-2$}.
\end{cases}\]
The chain complex $\tC_{\bullet}$ can be regarded as a chain complex of $\GL_n(K)$-modules,
but we will only consider it as a chain complex of $\GL_n(\cO)$-modules.
Define
$D_{\bullet} = \tC_{\bullet} \otimes \Q_{\chi}$, so
\[\HH_k\left(D_{\bullet}\right) = \begin{cases}
\St_n\left(K\right) \otimes \Q_{\chi} & \text{if $k=n-2$},\\
0 & \text{if $k \neq n-2$}
\end{cases}\]
as $\GL_n(\cO)$-modules.

We will examine the homology of $\GL_n(\cO)$ with coefficients in the chain complex
$D_{\bullet}$ in the sense of \cite[\S VII.5]{BrownCohomology}.  Letting
$F_{\bullet}$ be a projective resolution of the trivial $\GL_n(\cO)$-module $\Z$, by definition
$\HH_{\ast}(\GL_n(\cO);D_{\bullet})$ is the homology of the double complex
$F_{\bullet} \otimes D_{\bullet}$.  Just like for any double complex, there are two spectral
sequences converging to the homology of $F_{\bullet} \otimes D_{\bullet}$.  The
first spectral sequence has
\[E^2_{pq} = \HH_q\left(\GL_n\left(\cO\right);\HH_p\left(D_{\bullet}\right)\right) = \begin{cases}
\HH_q\left(\GL_n\left(\cO\right);\St_n\left(K\right) \otimes \Q_{\chi}\right) & \text{if $p=n-2$},\\
0 & \text{if $p \neq n-2$}.
\end{cases}\]
This spectral sequence thus degenerates to show that
\[\HH_k\left(\GL_n\left(\cO\right);D_{\bullet}\right) = \HH_{k-(n-2)}\left(\GL_n\left(\cO\right);\St_n\left(K\right) \otimes \Q_{\chi}\right).\]
We deduce that our goal is equivalent to showing that
$\HH_{n-2}(\GL_n(\cO);D_{\bullet}) = 0$.

The second spectral sequence converging to the homology of $F_{\bullet} \otimes D_{\bullet}$
has
\[(E')^1_{pq} = \HH_p\left(\GL_n\left(\cO\right);D_q\right).\]
To prove that $\HH_{n-2}(\GL_n(\cO);D_{\bullet}) = 0$, it is enough to
prove that $(E')^1_{pq} = \HH_p(\GL_n(\cO);D_q) = 0$ for all
$p \geq 0$ and $q \geq -1$ such that $p+q = n-2$. To that end, fix
such $p$ and $q$.

Let $\cF$ be the set of length-$q$ flags in $K^n$; by convention, for $q=-1$ the
set $\cF$ consists of the single empty flag.  The vector space $D_q$ thus consists of formal
$\Q$-linear combinations of elements of $\cF$, where $\GL_n(\cO)$ acts
on $\cF$ via its obvious action and on the coefficients $\Q$ via $\chi$.  Let
$I$ be a set of orbit representatives for the action of
$\GL_n(\cO)$ on $\cF$.  For $\fF \in I$, recall that $\GL_n(\cO,\fF)$ is the
$\GL_n(\cO)$-stabilizer of $\fF$.  We have
\[D_q = \bigoplus_{\fF \in I} \Ind_{\GL_n\left(\cO,\fF\right)}^{\GL_n\left(\cO\right)} \Q_{\chi},\]
so
\[\HH_p\left(\GL_n\left(\cO\right);D_q\right) = \bigoplus_{\fF \in I} \HH_p\left(\GL_n\left(\cO\right);\Ind_{\GL_n\left(\cO,\fF\right)}^{\GL_n\left(\cO\right)} \Q_{\chi}\right) = \bigoplus_{\fF \in I} \HH_p\left(\GL_n\left(\cO,\fF\right);\Q_{\chi}\right),\]
where the final isomorphism comes from Shapiro's Lemma.  It
is thus enough to prove that $\HH_p(\GL_n(\cO,\fF);\Q_{\chi}) = 0$ for
all $\fF \in I$.

Fix $\fF \in I$. Recall that $\CL_n(\cO,\fF)$ is the kernel of the restriction of
$\chi\colon \GL_n(\cO) \rightarrow \{\pm 1\}$ to $\GL_n(\cO,\fF)$.  Since
$\CL_n(\cO,\fF)$ is a finite-index normal subgroup of $\GL_n(\cO,\fF)$,
the existence of the transfer map shows that
\[\HH_p\left(\GL_n\left(\cO,\fF\right);\Q_{\chi}\right) = \left(\HH_p\left(\CL_n\left(\cO,\fF\right);\Q_{\chi}\right)\right)_{\GL_n\left(\cO,\fF\right)},\]
where the subscript indicates that we are taking the $\GL_n(\cO,\fF)$-coinvariants.
See \cite[Proposition III.10.4]{BrownCohomology} for more details.

We thus must show that these coinvariants vanish.  Since $p = n-2-q \leq n-1$ (with
equality precisely when $q=-1$), we can apply
Proposition \ref{proposition:flagactiontrivial} to deduce that
the action of $\GL_n(\cO,\fF)$ on $\HH_p(\CL_n(\cO,\fF);\Q)$ is trivial.  Using
this along with the fact that $\CL_n(\cO,\fF)$ acts trivially on $\Q_{\chi}$,
we compute as follows:
\begin{align*}
\left(\HH_p\left(\CL_n\left(\cO,\fF\right);\Q_{\chi}\right)\right)_{\GL_n\left(\cO,\fF\right)} &= \left(\HH_p\left(\CL_n\left(\cO,\fF\right);\Q\right) \otimes \Q_{\chi}\right)_{\GL_n\left(\cO,\fF\right)} \\
&= \HH_p\left(\CL_n\left(\cO,\fF\right);\Q\right) \otimes \left(\Q_{\chi}\right)_{\GL_n\left(\cO,\fF\right)} \\
&= \HH_p\left(\CL_n\left(\cO,\fF\right);\Q\right) \otimes 0 = 0.
\end{align*}
Here we are using the fact that $\cO$ has an element of norm $-1$, so the
group $\GL_n(\cO,\fF)$ acts nontrivially on $\Q_{\chi}$ and
$(\Q_{\chi})_{\GL_n(\cO,\fF)} = 0$.  The theorem follows.
\end{proof}

\section{Reduction II: splitting a flag}
\label{section:reductionii}

In the previous section, we reduced Theorem \ref{maintheorem:highsmall} to Proposition
\ref{proposition:flagactiontrivial}.  In this section, we reduce Proposition 
\ref{proposition:flagactiontrivial}
to two further propositions that will be proven in subsequent sections.

\subsection{Basic facts about flags}
\label{section:flagsbasics}

Before we can do this reduction, we must discuss some basic facts about flags for
which \cite{MilnorKTheory} is a suitable reference.  Let $\cO$ be the ring of
integers in a number field $K$.  Fix a finite-rank projective $\cO$-module $Q$ and
let $n = \rk(Q)$.  We can then identify $K^n$ with $Q \otimes K$.

\p{Subspace stabilizers and projective modules}
For a subspace $V$ of $K^n = Q \otimes K$, the intersection $V \cap Q$ is a direct summand
of $Q$.  Here is a quick proof of this standard fact: 
$Q / V \cap Q$ is a finitely generated $\cO$-submodule of $K^n/V$, and thus is 
torsion-free and hence projective, allowing
us to split the short exact sequence 
\[0 \longrightarrow V \cap Q \longrightarrow Q \longrightarrow Q / V \cap Q \longrightarrow 0.\]
This implies that $V \cap Q$ is itself a projective $\cO$-module.

\p{Splitting flag stabilizers}
Now consider a flag $\fF$ in $K^n = Q \otimes K$ of the form
\[0 \subsetneq \fF_0 \subsetneq \fF_1 \subsetneq \cdots \subsetneq \fF_q \subsetneq K^n.\]
Just like we did for $\GL_n(\cO)$, we will write $\GL(Q,\fF)$ for the subgroup of $\GL(Q)$ stabilizing
$\fF$.  Intersecting our flag with $Q$, we obtain a flag
\[0 \subsetneq \fF_0 \cap Q \subsetneq \fF_1 \cap Q \subsetneq \cdots \subsetneq \fF_q \cap Q \subsetneq Q\]
of direct summands of $Q$.  Each term of this flag is a direct summand of the next one.  Iteratively
splitting each off from the next, we obtain a decomposition 
\[Q = P_0 \oplus P_1 \oplus \cdots \oplus P_{q+1}\]
such that
\[\fF_i \cap Q = P_0 \oplus \cdots \oplus P_i \quad \quad (0 \leq i \leq q).\]
The $P_i$ are all projective $\cO$-modules, and we will call the sequence
$\fP = (P_0,\ldots,P_{q+1})$ a {\em projective splitting} of $Q$ with respect to
the flag $\fF$.  Define
\[\GL(Q,\fP) = \GL(P_0) \times \cdots \times \GL(P_{q+1}) \subset \GL(Q,\fF).\]
If $Q \cong \cO^n$, we will often write $\GL_n(\cO,\fP)$ instead of $\GL(\cO^n,\fP)$.

\p{Determinants of automorphisms of projective modules}
For a finite-rank projective $\cO$-module $P$, we have
\[\GL(P) \subset \GL(P \otimes K) \cong \GL_{\rk(P)}(K),\]
so there is a well-defined determinant map $\GL(P) \rightarrow K^{\times}$.  In fact,
the image of this map lies in $\cO^{\times}$:

\begin{lemma}
\label{lemma:detprojective}
Let $\cO$ be the ring of integers in a number field $K$ and let $P$ be a finite-rank
projective $\cO$-module.  Then $\det(f) \in \cO^{\times}$ for $f \in \GL(P)$.
\end{lemma}
\begin{proof}
Since $P$ is a finite-rank projective $\cO$-module, there exists another finite-rank projective
$\cO$-module $P'$ such that $P \oplus P' \cong \cO^m$ for some $m$.  Extending automorphisms of
$P$ over $P'$ by the identity, we get an embedding $\GL(P) \hookrightarrow \GL(\cO^m)$
that fits into a commutative diagram
\[\begin{CD}
\GL(P)     @>>> \GL(P \otimes K)     @>{\cong}>> \GL_{\rk(P)}(K) @>{\det}>> K^{\times} \\
@VVV            @VVV                             @VVV                       @VV{=}V    \\
\GL(\cO^m) @>>> \GL(\cO^m \otimes K) @>{\cong}>> \GL_m(K)        @>{\det}>> K^{\times}.
\end{CD}\]
We get an equality on the rightmost vertical arrow since, with respect to an appropriate basis,
the map $\GL_{\rk(P)}(K) \rightarrow \GL_m(K)$ is the standard one induced by the inclusion
$K^{\rk(P)} \hookrightarrow K^m$.  Since matrices in $\GL(\cO^n) \cong \GL_n(\cO)$ have
determinant in $\cO^{\times}$, so do matrices in $\GL(P)$.
\end{proof}

Assuming now that $\cO^{\times}$ has an element of norm $-1$, we can define
$\CL(P)$ to be the kernel of the map $\chi\colon \GL(P) \rightarrow \{\pm 1\}$ obtained
by composing the determinant with the norm map $\cO^{\times} \rightarrow \{\pm 1\}$.

\p{Splitting flag stabilizers II}
Continue to assume that $\cO^{\times}$ has an element of norm $-1$.  Recall that $Q$ is
a fixed rank-$n$ projective $\cO$-module.
If $\fF$ is a length-$q$ flag in $K^n = Q \otimes K$ and $\fP = (P_0,\ldots,P_{q+1})$ is a projective
splitting of $Q$ with respect to $\fF$, then using the above we can define
\[\CL(Q,\fP) = \CL(P_0) \times \cdots \times \CL(P_{q+1}) \subset \GL(Q,\fP).\]
The group $\CL(Q,\fP)$ is a normal subgroup of $\GL(Q,\fP)$ of index $2^{q+2}$ (see
Remark \ref{remark:achievedet} below if this is not clear).  Just like
above, for $Q = \cO^n$ we will sometimes write $\CL_n(\cO,\fP)$ instead of $\CL(Q,\fP)$.

\begin{remark}
\label{remark:achievedet}
If $\fF$ is a flag in $K^n$ (possibly the empty flag), then the determinant
map $\GL(Q,\fF) \rightarrow \cO^{\times}$ is surjective (and thus
if $\cO$ has an element of norm $-1$, then $\CL(Q,\fF)$ is an index-$2$ subgroup
of $\GL(Q,\fF)$).  Indeed, without loss
of generality we can assume that $\fF$ is a maximal flag since this just replaces
$\GL(Q,\fF)$ by a subgroup.  Let $\fP = (P_0,\ldots,P_{n-1})$ be a projective
splitting of $Q$ with respect to $\fF$, so we have
\[\GL(Q,\fP) = \GL(P_0) \times \cdots \times \GL(P_{n-1}) \subset \GL(Q,\fF).\]
For all $d \in \cO^{\times}$, the element of $\GL(Q,\fP)$ that scales
$P_0$ by $d$ and fixes $P_1,\ldots,P_{n-1}$ lies in $\GL(Q,\fF)$ and has determinant
$d$.
\end{remark}

\subsection{The reduction}
\label{section:reductioniiwork}

We now turn to Proposition \ref{proposition:flagactiontrivial}.  Our goal is
to reduce it to two propositions.  The first is the following, which
informally says in a range of degrees the homology groups of a flag-stabilizer
are completely supported on a projective splitting:

\begin{proposition}
\label{proposition:carry}
Let $\cO$ be the ring of integers in a number field $K$, let $Q$ be a rank-$n$
projective $\cO$-module, let
$\fF$ be a flag in $Q \otimes K$, and let $\fP$ be a projective splitting of $Q$ with respect
to $\fF$.
Assume that $\cO^{\times}$ has an element of norm $-1$, and
let $r$ and $2s$ be the numbers of real and complex embeddings of $K$.
Then the map $\HH_k(\CL(Q,\fP);\Q) \rightarrow \HH_k(\CL(Q,\fF);\Q)$
is a surjection for $0 \leq k \leq r+s-1$.
\end{proposition}

The second is the following, which is 
a generalization from $\cO^n$ to an arbitrary finite-rank
projective module of the special case of Proposition \ref{proposition:flagactiontrivial}
where the flag is trivial, and thus the conclusion of Proposition \ref{proposition:flagactiontrivial} is
that $\GL_n(\cO)$ acts trivially on the rational homology of $\CL_n(\cO)$.

\begin{proposition}
\label{proposition:projectiveactiontrivial}
Let $\cO$ be the ring of integers in a number field $K$ and let
$P$ be a finite-rank projective $\cO$-module.
Assume that $\cO^{\times}$ has an element of norm $-1$, and
let $r$ and $2s$ be the numbers of real and complex embeddings of $K$.
Then the action of $\GL(P)$ on its normal subgroup $\CL(P)$ induces
the trivial action on $\HH_k(\CL(P);\Q)$ for
$0 \leq k \leq \min(r+s,\rk(P))-1$.
\end{proposition}

We will prove Propositions \ref{proposition:carry} and \ref{proposition:projectiveactiontrivial} in \S \ref{section:splitflag} and \S \ref{section:trivial}, respectively. 
Here we will assume their truth and derive
Proposition \ref{proposition:flagactiontrivial}.

\begin{proof}[Proof of Proposition \ref{proposition:flagactiontrivial}, assuming
Propositions \ref{proposition:carry} and \ref{proposition:projectiveactiontrivial}]
Let us recall the setup.  Let $\cO$ be the ring of integers in a number field $K$ and let
$\fF$ be a flag in $K^n$.  Assume that $\cO^{\times}$ contains an element of
norm $-1$, and let $r$ and $2s$ be the numbers of real and complex embeddings
of $K$.  Consider some $0 \leq k \leq \min(r+s,n)-1$.
We must prove that
the action of $\GL_n(\cO,\fF)$ on its normal subgroup $\CL_n(\cO,\fF)$ induces
the trivial action on $\HH_k(\CL_n(\cO,\fF);\Q)$. \setlength{\parskip}{\baselineskip}

Let $\fP = (P_0,\ldots,P_m)$ be a projective splitting of $\cO^n$ with respect to $\fF$.  By Proposition 
\ref{proposition:carry}, the map 
\[\HH_k(\CL_n(\cO,\fP);\Q) \rightarrow \HH_k(\CL_n(\cO,\fF);\Q)\]
is surjective.  The K\"{u}nneth formula says that
\begin{align}
\HH_k(\CL_n(\cO,\fP);\Q) &= \HH_k(\CL(P_0) \times \cdots \times \CL(P_m);\Q) \notag\\
&\cong \bigoplus_{i_0+\cdots+i_m=k} \HH_{i_0}(\CL(P_0);\Q) \otimes \cdots \otimes \HH_{i_m}(\CL(P_m);\Q). \label{eqn:kunneth}
\end{align}
It is thus enough to show that $\GL(\cO,\fF)$ acts trivially on the images of each
of these factors in $\HH_k(\CL_n(\cO,\fF);\Q)$.

Consider a summand
\[V = \HH_{i_0}(\CL(P_0);\Q) \otimes \cdots \otimes \HH_{i_m}(\CL(P_m);\Q)\]
of \eqref{eqn:kunneth}.  Since inner automorphisms always act trivially on homology
and $\CL_n(\cO,\fF)$ is an index-$2$ subgroup of $\GL_n(\cO,\fF)$, it is enough
to find a single element of $\GL_n(\cO,\fF) \setminus \CL_n(\cO,\fF)$ that acts trivially
on the image of $V$ in $\HH_k(\CL_n(\cO,\fF);\Q)$.  Since
\[i_0+\cdots+i_m = k \leq \min(r+s,n-1) \quad \text{and} \quad \rk(P_0)+\cdots+\rk(P_m) = n,\]
there must exist some $0 \leq j \leq m$ such that $i_j \leq \min(r+s,\rk(P_j))-1$.  We can thus
apply Proposition \ref{proposition:projectiveactiontrivial} to see that
$\GL(P_j)$ acts trivially on $\HH_{i_j}(\CL(P_j);\Q)$.  Pick $x_j \in \GL(P_j)$
such that $x_j \notin \CL(P_j)$ (and thus $\chi(x_j)=-1$).  For $0 \leq j' \leq m$
with $j' \neq j$, set $x_{j'} = 1 \in \GL(P_{j'})$.  Set 
\[x = (x_0,\ldots,x_m) \in \GL(P_0) \times \cdots \times \GL(P_m) = \GL_n(\cO,\fP).\]
We thus have $x \notin \CL_n(\cO,\fF)$, and by construction $x$ acts trivially
on the image of $V$ in $\HH_k(\CL_n(\cO,\fP);\Q)$ and hence also on the image
of $V$ in $\HH_k(\CL_n(\cO,\fF);\Q)$, as desired.
\end{proof}

\section{The homology carried on a split flag}
\label{section:splitflag}

In this section, we will prove Proposition \ref{proposition:carry}.  
We start in \S \ref{section:decompose} with a basic structural result about flag
stabilizers, and then in \S \ref{section:carry} we reduce the proof to a
simpler homological lemma whose proof occupies the remaining subsections of
this section.

\subsection{Decomposing stabilizers of flags}
\label{section:decompose}

Let $\cO$ be the ring of integers in an algebraic number field $K$ and let
$Q$ be a finite-rank projective $\cO$-module.  Proposition \ref{proposition:carry} 
concerns the homology of the $\GL(Q)$-stabilizer of a flag.  This section
shows how to decompose this stabilizer as a semidirect product.

\p{Motivating example}
To understand the form this decomposition takes, we start with a familiar example.
Let $\Gamma \subset \GL_{n+n'}(\R)$ be the subgroup consisting of matrices
with an $n' \times n$ block of zeros in their lower left hand corner:
\[\Gamma = \Set{$\left(
\begin{array}{@{}c|c@{}} 
A & \ast \\
\hline
0 & B
\end{array}\right)$}{$A \in \GL_n(\R)$ and $B \in \GL_{n'}(\R)$}.\]
The group $\Gamma$ contains the subgroups
\[\GL_n(\R) \times \GL_{n'}(\R) = \Set{$\left(
\begin{array}{@{}c|c@{}} 
A & 0 \\
\hline
0 & B
\end{array}\right)$}{$A \in \GL_n(\R)$ and $B \in \GL_{n'}(\R)$}\]
and
\[\Mat_{n,n'}(\R) = \Set{$\left(
\begin{array}{@{}c|c@{}} 
1 & U \\
\hline
0 & 1
\end{array}\right)$}{$U \in \Mat_{n,n'}(\R)$}.\]
The additive subgroup $\Mat_{n,n'}(\R)$ is normal, and
\begin{equation}
\label{eqn:togeneralize}
\Gamma = \Mat_{n,n'}(\R) \rtimes (\GL_n(\R) \times \GL_{n'}(\R)).
\end{equation}
The action of $\GL_n(\R) \times \GL_{n'}(\R)$ on $\Mat_{n,n'}(\R)$ in
\eqref{eqn:togeneralize} arises from the identification
$\Mat_{n,n'}(\R) = \Hom(\R^{n'},\R^n)$.

\p{Our decomposition}
Our analogue of \eqref{eqn:togeneralize} is as follows:

\begin{lemma}
\label{lemma:decompose}
Let $\cO$ be the ring of integers in a number field $K$, let
$Q$ be a rank-$n$ projective $\cO$-module, let
$\fF$ be a flag in $Q \otimes K = K^n$, and let $\fP=(P_0,\ldots,P_{t})$ 
be a projective splitting of $Q$ with respect to $\fF$.  Set $Q' = P_0 \oplus \cdots \oplus P_{t-1}$,
so $Q = Q' \oplus P_{t}$, and let $\fF'$ be the flag in $Q' \otimes K$
obtained by omitting the last term of $\fF$.  Then
$\GL(Q,\fF) = \Hom(P_{t},Q') \rtimes (\GL(Q',\fF') \times \GL(P_{t}))$.
\end{lemma}
\begin{proof}
Elements of $\GL(Q,\fF)$ preserve $Q'$, and thus also act on $Q/Q'$, which
we can identify with $P_t$.  
Combining the resulting homomorphisms $\GL(Q,\fF) \rightarrow \GL(Q',\fF')$
and $\GL(Q,\fF) \rightarrow \GL(P_t)$, we get a homomorphism
$\phi\colon \GL(Q,\fF) \rightarrow \GL(Q',\fF') \times \GL(P_t)$.  The homomorphism
$\phi$ is a split surjection via the evident inclusion
$\GL(Q',\fF') \times \GL(P_t) \hookrightarrow \GL(Q,\fF)$.  Letting $U = \ker(\phi)$,
we have $U \cong \Hom(P_t,Q')$ via the identification that takes
$f\colon P_t \rightarrow Q'$ to the automorphism of $Q$ taking
$(x,y) \in Q' \oplus P_t = Q$ to $(x+f(y),y)$.  The lemma follows.
\end{proof}

\subsection{A reduction}
\label{section:carry}

In this section, we reduce Proposition \ref{proposition:carry} to the following lemma.
For later use, we state the lemma in more generality than we need.

\begin{lemma}
\label{lemma:spectralsequence}
Let $\cO$ be the ring of integers in a number field $K$ and let $Q$ and $P$
be finite-rank projective $\cO$-modules.  Assume that $\cO$ contains an element
of norm $-1$, and let $r$ and $2s$ be the numbers of real and complex
embeddings of $K$.  Let $G$ be an arbitrary subgroup of $\GL(Q)$ and let
$\Gamma = \Hom(P,Q) \rtimes (G \times \CL(P))$.  Then the map
$\HH_k(G \times \CL(P);\Q) \rightarrow \HH_k(\Gamma;\Q)$ is an isomorphism
for $0 \leq k \leq r+s-1$.
\end{lemma}

The restriction on $k$ in the statement of Lemma \ref{lemma:spectralsequence} is the reason 
for the restriction on $n$ in Theorem \ref{maintheorem:highsmall}.
The proof of Lemma \ref{lemma:spectralsequence} occupies the remaining subsections of this
section.  Here we show how to derive Proposition \ref{proposition:carry} from it.

\begin{proof}[Proof of Proposition \ref{proposition:carry}, assuming Lemma \ref{lemma:spectralsequence}]
We first recall the setup.  Let $\cO$ be the ring of integers in a number field $K$, let
$Q$ be a rank-$n$ projective $\cO$-module, let
$\fF$ be a flag in $Q \otimes K = K^n$, and let $\fP$ be a projective splitting of $Q$ with respect to $\fF$.
Assume that $\cO^{\times}$ has an element of norm $-1$, and
let $r$ and $2s$ be the numbers of real and complex embeddings of $K$.  We must prove that
the map $\HH_k(\CL(Q,\fP);\Q) \rightarrow \HH_k(\CL(Q,\fF);\Q)$
is a surjection for $0 \leq k \leq r+s-1$.

Write $\fP = (P_0,\ldots,P_{t})$.  The proof will be by induction on $t$.
The base case $t=0$ being trivial,
assume that $t \geq 1$ and that the result is true for all smaller $t$.
Let $Q' = P_0 \oplus \cdots \oplus P_{t-1}$, so
$Q = Q' \oplus P_t$.  Let $\fF'$ be the flag in $Q' \otimes K$ obtained
by omitting the last term of $\fF$ and let $\fP'=(P_0,\ldots,P_{t-1})$, so
$\fP'$ is a projective splitting of $Q'$ with respect to $\fF'$.

Lemma \ref{lemma:decompose} says that
\begin{equation}
\label{eqn:breakgl}
\GL(Q,\fF) = \Hom(P_{t},Q') \rtimes (\GL(Q',\fF') \times \GL(P_{t})).
\end{equation}
We factor the map $\CL(Q,\fP) \rightarrow \CL(Q,\fF)$ as follows:
\begin{align*}
\CL(Q,\fP) &= \CL(Q',\fP') \times \CL(P_t) 
\stackrel{\phi_1}{\hookrightarrow} \CL(Q',\fF') \times \CL(P_t)\\
&\stackrel{\phi_2}{\hookrightarrow} \Hom(P_t,Q') \rtimes (\CL(Q',\fF') \times \CL(P_t))
\stackrel{\phi_3}{\hookrightarrow} \CL(Q,\fF).
\end{align*}
The map $\phi_3$ comes from identifying the indicated semidirect product with a subgroup of $\CL(Q,\fF)$ 
via \eqref{eqn:breakgl}.  It is enough to prove
that each $\phi_i$ induces a surjection on $\HH_k(-;\Q)$ for
$0 \leq k \leq r+s-1$:
\setlength{\parskip}{0pt}
\begin{compactitem}
\item For $\phi_1$, this comes from combining the K\"{u}nneth formula with our inductive
hypothesis, which implies that the map
$\HH_k(\CL(Q',\fP');\Q) \rightarrow \HH_k(\CL(Q',\fF');\Q)$ is a surjection
for $0 \leq k \leq r+s-1$.
\item For $\phi_2$, this follows from Lemma \ref{lemma:spectralsequence}.
\item For $\phi_3$, this follows from the fact that $\phi_3$ is the inclusion
of a finite-index subgroup and thus induces a surjection on $\HH_k(-;\Q)$
for all $k$, which follows from the existence of the transfer map (see, e.g.\
\cite[\S III.9]{BrownCohomology}).\qedhere
\end{compactitem}
\end{proof}
\setlength{\parskip}{\baselineskip}

\subsection{Killing homology with a center}
\label{section:kill}

We will prove Lemma \ref{lemma:spectralsequence} by studying the Hochschild--Serre
spectral sequence of the indicated semidirect product.
This spectral sequence is composed of various
twisted homology groups, and our goal will be to show that most of them vanish.  The
following lemma gives a simple criterion for showing this.
 
\begin{lemma}
\label{lemma:kill}
Let $G$ be a group and let $M$ be a finite-dimensional vector space over a field of characteristic $0$ upon which $G$ acts.
Assume that there exists a central element $c$ of $G$ that fixes no nonzero element of $M$.  Then
$\HH_k(G;M)=0$ for all $k$.
\end{lemma}
\begin{proof}
Let $C$ be the cyclic subgroup of $G$ generated by $c$.  Since $c$ is central, the subgroup $C$ is central and
hence normal in $G$.  Define $Q=G/C$.  We thus have a short exact sequence
\[1 \longrightarrow C \longrightarrow G \longrightarrow Q \longrightarrow 1.\]
The associated Hochschild--Serre spectral sequence is of the form
\[E^2_{pq} = \HH_p(Q;\HH_q(C;M)) \Rightarrow \HH_{p+q}(G;M).\]
To prove that $\HH_k(G;M)=0$ for all $k$, it is enough to prove that all terms of this spectral
sequence vanish.  In fact, we will prove that $\HH_q(C;M)=0$ for all $q$.

Since $c$ fixes no nonzero element of $M$, the linear map $M \rightarrow M$ taking $x \in M$ to $cx-x \in M$ has a trivial
kernel.  It is thus an isomorphism, which immediately implies that the $C$-coinvariants
$\HH_0(C;M) = M_C$ vanish.  If $c$ has finite order, then $C$ is a finite group.  Since $M$ is
a vector space over a field of characteristic $0$, this implies that $\HH_q(C;M)=0$ for all $q \geq 1$, and
we are done.  Otherwise, $C \cong \Z$ and we also have to check that $\HH_1(C;M)=0$.  For this, we apply
Poincar\'{e} duality to $\Z$ (the fundamental group of a circle!) to see that
$\HH_1(C;M) \cong \HH^0(C;M) = M^C$.  These invariants vanish by assumption.
\end{proof}

The following lemma will help us recognize when Lemma \ref{lemma:kill} applies.

\begin{lemma}
\label{lemma:recognize}
Let $C$ be a group and let $M$ be a finite-dimensional vector space on which $C$ acts.  Let
$\phi\colon C \rightarrow \GL(M)$ be the associated homomorphism and let $\oC \subset \GL(M)$
be the Zariski closure of $\phi(C)$.  Assume that $\oC$ contains an element that fixes
no nonzero element of $M$.  Then $C$ does as well.
\end{lemma}
\begin{proof}
The set of $x \in \GL(M)$ that fix a nonzero element of $M$ is a Zariski-closed subspace; indeed,
it is precisely the set of all $x$ such that $\det(x-1)=0$.  By assumption, $\oC$ is not contained in it,
so $\phi(C)$ must not be as well.
\end{proof}

\subsection{The Zariski closure of units}
\label{section:zariski}

To apply Lemma \ref{lemma:kill} to the Hochschild--Serre spectral sequence associated
to the split short exact sequence
\[1 \longrightarrow \Hom(P,Q) \longrightarrow \Gamma \longrightarrow G \times \CL(P) \longrightarrow 1\]
discussed in Lemma \ref{lemma:spectralsequence}, we need some interesting central elements
of $G \times \CL(P)$.  Let $\cO^{\times}_1$ be the set of norm-$1$ units in $\cO$.  The
central elements we will use are in the subgroup $\cO^{\times}_1$ of $\CL(P)$, which acts
on $P$ as scalar multiplication.

We will want to apply Lemma \ref{lemma:recognize} to this, which requires identifying
the Zariski closure of $\cO^{\times}_1$ in an appropriate real
algebraic group.
To state the general result we will prove, let $r$ and $2s$ be the numbers
of real and complex embeddings of the algebraic number field $K$, 
so $\cO \otimes \R \cong \R^r \oplus \C^s$, where $\C^s$
is regarded as a $2s$-dimensional $\R$-vector space.  The group $\cO^{\times}$ acts on
$\cO \otimes \R$, providing us with a representation
\[\cO^{\times} \longrightarrow \GL(\cO \otimes \R) \cong \GL_{r+2s}(\R).\]
The following lemma identifies the Zariski closure of the image of 
$\cO^{\times}_1$ in $\GL(\cO \otimes \R)$ when $\cO^{\times}$ has an
element of norm $-1$, since any such $K$ has a real embedding.  

\begin{lemma}
\label{lemma:identifyclosure}
Let $\cO$ be the ring of integers in an algebraic number field $K$.  Assume that $K$ has a real embedding, and
let $r$ and $2s$ be the numbers of real and complex embeddings of $K$, so $\cO \otimes \R \cong \R^r \oplus \C^{s}$.  
The Zariski closure of the image of $\cO^{\times}_1$ in $\GL(\cO \otimes \R) \cong \GL_{r+2s}(\R)$ is
\[\SetLong{(a_1,\ldots,a_r,b_1,\ldots,b_s) \in (\R^{\times})^r \times (\C^{\times})^s}{\prod_{j=1}^r a_j \prod_{k=1}^s |b_k| = 1}.\]
\end{lemma}

\begin{remark}
Lemma \ref{lemma:identifyclosure} is not true for all algebraic number fields.  For instance, the norm-$1$ units
in $\Z[i]$ are $\{\pm 1, \pm i\}$, which are not Zariski dense in $\Set{$b \in \C^{\times}$}{$|b|=1$}$.  It turns
out that the conclusion of Lemma \ref{lemma:identifyclosure} holds if and only if $K$ does not contain a CM subfield.
We will not need this stronger result, so we prove only the above for the sake of brevity.
\end{remark}

Lemma \ref{lemma:identifyclosure} could be deduced from general results about algebraic tori (see, e.g.,
\cite[Appendix to Chapter 2]{SerreLadic}).  To make this paper more self-contained, we include
an elementary proof.  We would like to thank Will Sawin for showing it to us.  

Our proof will require a consequence of the standard proof of the Dirichlet Unit Theorem (see,
e.g., \cite[\S I.7]{NeukirchBook}).  Continuing the above notation, let
$f_1,\ldots,f_r \colon K \rightarrow \R$ and  
$f_{r+1},\overline{f}_{r+1},\ldots,f_{r+s},\overline{f}_{r+s}\colon K \rightarrow \C$ be the real
and complex embeddings of $K$.  For $x \in K$, the norm of $x$ equals
\begin{equation}
\label{eqn:normproduct}
f_1(x) \cdots f_r(x) \cdot |f_{r+1}|^2 \cdots |f_{r+s}|^2.
\end{equation}
For $x \in \cO^{\times}$, this will be $\pm 1$.  To convert the multiplication
in $\cO^{\times}$ into addition and also to eliminate the distinction between $\pm 1$,
we take absolute values and logarithms, and define $\Phi\colon K^{\times} \rightarrow \R^{r+s}$
via the formula
\[\Psi(x) = (\log |f_1(x)|,\ldots,\log |f_r(x)|, 2 \log |f_{r+1}(x)|,\ldots,2 \log |f_{r+s}(x)|) \in \R^{r+s}.\]
For $x \in \cO^{\times}$, the fact that \eqref{eqn:normproduct} is $\pm 1$ implies that
\[\log |f_1(x)| + \cdots + \log |f_r(x)| + 2 \log|f_{r+1}(x)| + \cdots + 2 \log|f_{r+s}(x)| = 0.\]
In other words, for $x \in \cO^{\times}$ the image $\Psi(x) \in \R^{r+s}$ lies
in the hyperplane
\[H = \Set{$(x_1,\ldots,x_{r+s})$}{$x_1 + \cdots + x_{r+s} = 0$}.\]
The key step in the proof of the Dirichlet Unit Theorem is showing that
$\Psi(\cO^{\times})$ is a lattice in $H$; see \cite[Theorem I.7.3]{NeukirchBook}.
Since the norm-$1$ units $\cO_1^{\times}$ are an index-$2$ subgroup of $\cO^{\times}$,
it follows that $\Psi(\cO_1^{\times})$ also forms a lattice in $H$.  As a
consequence, we deduce the following.

\begin{lemma}
\label{lemma:unittheorem}
Letting the notation be as above, consider $c_1,\ldots,c_{r+s} \in \R$ such
that
\[c_1 \log|f_1(x)| + \cdots + c_{r+s} \log|f_{r+s}(x)| = 0 \quad \quad \text{for all $x \in \cO_1^{\times}$}.\]
Then $2 c_1 = \cdots = 2 c_r = c_{r+1} = \cdots = c_{r+s}$.
\end{lemma}

\begin{proof}[Proof of Lemma \ref{lemma:identifyclosure}]
Let $f_1,\ldots,f_{r+2s}\colon \cO^{\times}_1 \rightarrow \omQ^{\times}$ 
be the restrictions to $\cO^{\times}_1$ of the different embeddings of
$K$ into $\omQ$, ordered in an arbitrary way.  
The norm of an element of $K$ is the product of its images
under the different embeddings of $K$ into $\omQ$, so since $\cO^{\times}_1$
consists of elements of norm $1$ we have $f_1 \cdots f_{r+2s} = 1$.
Let $\Lambda$ be the $\R$-algebra of $\C$-valued functions on $\cO^{\times}_1$.
Fix an embedding $\omQ \hookrightarrow \C$ and
let $\phi\colon \R[x_1^{\pm 1},\ldots,x_{r+2s}^{\pm 1}] \rightarrow \Lambda$ be the
algebra map taking $x_i$ to $f_i$.  We have $x_1 \cdots x_{r+2s}-1 \in \ker(\phi)$,
and the lemma is equivalent to the assertion that the
ideal $I$ in $\R[x_1^{\pm 1},\ldots,x_{r+2s}^{\pm 1}]$ generated by
$x_1 \cdots x_{r+2s}-1$ equals $\ker(\phi)$.  Note that this is
independent of the order of the embeddings $f_1,\ldots,f_{r+2s}$.

The starting point is the following special case.

\begin{claim}
Let $m \in \R[x_1^{\pm 1},\ldots,x_{r+2s}^{\pm 1}]$ be a monomial such that
$m-1 \in \ker(\phi)$.  Then $m-1 \in I$.
\end{claim}
\begin{proof}[Proof of claim]
Write $m = x_1^{d_1} \cdots x_{r+2s}^{d_{r+2s}}$ with each $d_i \in \Z$.  To
prove the claim, it is enough to prove that all the $d_i$ are equal.  Since
$m-1 \in \ker(\phi)$, the function
\[\phi(m) = f_1^{d_1} \cdots f_{r+2s}^{d_{r+2s}}\colon \cO^{\times}_1 \rightarrow \omQ\]
is the trivial character.
Reordering the $f_i$ if necessary, we can assume that $d_1 \geq d_i$ for
all $1 \leq i \leq r+2s$.  Since there is at least one real embedding of $K$,
we can change our embedding $\omQ \hookrightarrow \C$ by precomposing
it with an appropriate element of the absolute Galois group and
ensure that $f_1$ is a real embedding.
We finally reorder $f_2,\ldots,f_{r+2s}$
such that $f_1,\ldots,f_r$ are the real embeddings, such that $f_{r+1},\ldots,f_{r+2s}$
are the complex embeddings,
and such that $\overline{f}_{r+i} = f_{r+i+s}$ for all $1 \leq i \leq s$.

For $u \in \cO^{\times}_1$, we have $f_1^{d_1}(u) \cdots f_{r+2s}^{d_{r+2s}}(u)=1$, so
\begin{align*}
1 &= \left(f_1^{d_1}\left(u\right) \cdots f_{r+2s}^{d_{r+2s}}\left(u\right)\right) \overline{\left(f_1^{d_1}\left(u\right) \cdots f_{r+2s}^{d_{r+2s}}\left(u\right)\right)} \\
&= |f_1(u)|^{2d_1} \cdots |f_r(u)|^{2d_r} |f_{r+1}(u)|^{2d_{r+1}+2d_{r+1+s}} \cdots |f_{r+s}|^{2d_{r+s} + 2d_{r+2s}}.
\end{align*}
Taking logarithms and dividing by $2$, we see that
\begin{align*}
0 &= d_1 \log|f_1(u)| + \cdots + d_r \log|f_r(u)| \\
&\quad\quad +(d_{r+1}+d_{r+1+s}) \log|f_{r+1}(u)|
+ \cdots + (d_{r+s}+d_{r+2s}) \log|f_{r+s}(u)|
\end{align*}
for all $u \in \cO^{\times}_1$.  Lemma \ref{lemma:unittheorem} then implies
that
\[2d_1 = \cdots = 2d_r = d_{r+1}+d_{r+1+s} = \cdots = d_{r+s}+d_{r+2s}.\]
Since $d_1 \geq d_{r+i}$ and $d_1 \geq d_{r+i+s}$, the only way that we can
have $d_{r+i}+d_{r+i+s} = 2d_1$ is for $d_{r+i} = d_{r+i+s} = d_1$, so in fact
\[d_1 = d_2 = \cdots = d_{r+2s},\]
as desired.
\end{proof}

We now turn to the general case.  Consider a nonzero $\theta \in \ker(\phi)$.  Write 
\[\theta = \sum_{i=1}^k \lambda_i m_i \quad \text{with $m_i \in \R[x_1^{\pm 1},\ldots,x_{r+2s}^{\pm 1}]$ a monomial and $\lambda_i \in \R$}.\]
Collecting terms, we can assume that the $m_i$ are all distinct and that $\lambda_i \neq 0$ for all $i$.  
For $1 \leq i \leq k$, the image $\phi(m_i)\colon \cO^{\times}_1 \rightarrow \C$ is a character.  Since
distinct characters on an abelian group are linearly independent, it follows
that there are distinct $1 \leq j,j' \leq k$ such that $\phi(m_j) = \phi(m_{j'})$.
This implies that $\phi(m_j m_{j'}^{-1})$ is the trivial character, so $m_j m_{j'}^{-1}-1 \in \ker(\phi)$.
The above claim thus implies that $m_j m_{j'}^{-1}-1 \in I$, so
\[m_j - m_{j'} = m_{j'}(m_j m_{j'}^{-1}-1) \in I.\]
Subtracting $\lambda_j(m_j - m_{j'}) \in I$ from $\theta$ eliminates its $\lambda_j m_j$ term.  
Collecting terms in $\theta$ and repeating the above argument over and over again, we conclude
that $\theta \in I$, as desired.
\end{proof}

\subsection{The proof of Lemma \ref{lemma:spectralsequence}}
\label{section:spectralsequence}

We finally prove Lemma \ref{lemma:spectralsequence}.

\begin{proof}[Proof of Lemma \ref{lemma:spectralsequence}]
We start by recalling what want to prove.  Let $\cO$ be the ring of integers
in a number field $K$ and let $Q$ and $P$ be finite-rank projective $\cO$-modules.
Assume that $\cO$ contains an element
of norm $-1$, and let $r$ and $2s$ be the numbers of real and complex
embeddings of $K$.  Let $G$ be an arbitrary subgroup of $\GL(Q)$ and let
$\Gamma = \Hom(P,Q) \rtimes (G \times \CL(P))$.  Our goal
is to prove that the map
$\HH_k(G \times \CL(P);\Q) \rightarrow \HH_k(\Gamma;\Q)$ is an isomorphism
for $0 \leq k \leq r+s-1$.  It is a little easier (but equivalent) to prove this
with real coefficients.

The Hochschild--Serre spectral sequence for the split extension
\[1 \longrightarrow \Hom(P,Q) \longrightarrow \Gamma \longrightarrow G \times \CL(P) \longrightarrow 1\]
is of the form
\[E^2_{pq} = \HH_p(G \times \CL(P);\HH_q(\Hom(P,Q);\R)) \Rightarrow \HH_{p+q}(\Gamma;\R).\]
We have
\[E^2_{k0} = \HH_k(G \times \CL(P);\R),\]
and to prove the lemma it is enough to prove that $E^2_{pq}=0$ for all $p$ and all
$1 \leq q \leq r+s-1$.  Fix some $1 \leq q \leq r+s-1$.
The group $\CL(P)$ contains the central subgroup $\cO^{\times}_1$, which
acts on $P$ as scalar multiplication.  Combining Lemmas \ref{lemma:kill} and \ref{lemma:recognize},
it is enough to prove that the Zariski closure of the image of $\cO^{\times}_1$ in the group
$\GL(\HH_q(\Hom(P,Q);\R))$ contains an element that fixes no nonzero vector of
$\HH_q(\Hom(P,Q);\R)$.

We have
\[\HH_q\left(\Hom\left(P,Q\right);\R\right) \cong \wedge^q \left(\Hom\left(P,Q\right) \otimes \R\right).\]
We now identify $\Hom(P,Q) \otimes \R$:

\begin{claim}
Let $n = \rk(P)$ and $m=\rk(Q)$.  We then have
\[\Hom(P,Q) \otimes \R \cong \Mat_{n,m}(\cO \otimes \R).\]
\end{claim}
\begin{proof}[Proof of claim]
By the classification of finitely generated projective modules over Dedekind domains (see, 
e.g.\ \cite[\S 1]{MilnorKTheory}), there exist nonzero ideals $I,J \subset \cO$ such
that $P = \cO^{n-1} \oplus I$ and $Q = \cO^{m-1} \oplus J$.  Using this identification, we see
$\Hom(P,Q)$ can be viewed as
\[\Set{$\left(
\begin{array}{@{}c|c@{}}
A & B \\
\hline
C & D
\end{array}\right)$}{$A \in \Mat_{m-1,n-1}(\cO)$, $B \in \Mat_{m-1,1}(I^{-1})$, $C \in \Mat_{1,n-1}(J)$, $D \in J I^{-1}$}.\]
Here $I^{-1} \subset K$ is the inverse of $I$ using the usual multiplication of fractional ideals
in a Dedekind domain.  The claim now follows from the fact that
\[\cO \otimes \R = J \otimes \R = I^{-1} \otimes \R = J I^{-1} \otimes \R = K \otimes \R.\qedhere\]
\end{proof}

From this, we see that the action of $\cO^{\times}_1$ on $\wedge^q(\Hom(P,Q) \otimes \R)$ can be
identified with the action of $\cO^{\times}_1$ on 
\[V:=\wedge^q (\cO \otimes \R)^{nm} \cong \wedge^q (\R^r \oplus \C^s)^{nm}.\]
Identify $\GL(\R^r\oplus \C^s)$ as a Zariski-closed subgroup of
$\GL(V)$ in the natural way.
By Lemma \ref{lemma:identifyclosure}, the Zariski closure of the image of $\cO^{\times}_1$ in 
$\GL(V)$ can be identified with
\begin{equation}
\label{eqn:theclosure}
\Set{$(a_1,\ldots,a_r,b_1,\ldots,b_s) \in (\R^{\times})^r \times (\C^{\times})^s$}{$\prod_{j=1}^r a_j \prod_{k=1}^s |b_k| = 1$},
\end{equation}
which acts on $\R^r \oplus \C^s$ by scalar multiplication.  We claim that the element
\[x = (a_1,\ldots,a_r,b_1,\ldots,b_s) = (2,\ldots,2,\frac{1}{2^{r+s-1}})\]
fixes no nonzero vector in $V$.  Indeed, the eigenvalues for the action
of $x$ on $V$ lie in the set of elements that can be expressed as the product of
$q$ elements of $\{2,\frac{1}{2^{r+s-1}}\}$, and $q \leq r+s-1$, so $1$ cannot be
expressed in this form.
\end{proof}

\section{The action on automorphisms of projectives is trivial}
\label{section:trivial}

In this section, we prove Proposition \ref{proposition:projectiveactiontrivial}.  The actual
proof is in \S \ref{section:projectiveactiontrivialproof}.  This is preceded by two sections
of preliminary results.

\subsection{Equivariant homology}
\label{section:equivarianthomology}

Our proof of Proposition \ref{proposition:projectiveactiontrivial} will use
a bit of equivariant homology.
In this section, we review some standard facts about this.
See \cite[\S VII.7]{BrownCohomology} for a textbook reference.

\p{Semisimplicial sets}
The natural setting for our proof is that of semisimplicial sets, which are a technical
variant on simplicial complexes whose definition we briefly recall.  For more
details, see \cite{FriedmanSimplicial}, which calls them $\Delta$-sets.
Let $\Delta$ be the category with objects the sets $[k]=\{0,\ldots,k\}$ for $k \geq 0$
and whose morphisms $[k] \rightarrow [\ell]$ are the strictly increasing functions.
A {\em semisimplicial set} is a contravariant functor $X$ from $\Delta$ to the
category of sets.  The $k$-simplices of $X$ are the image $X^{(k)}$ of $[k] \in \Delta$.
The maps $X^{(\ell)} \rightarrow X^{(k)}$ corresponding to the $\Delta$-morphisms
$[k] \rightarrow [\ell]$ are called the {\em boundary maps}.

\p{Geometric properties}
A semisimplicial set $X$ has a geometric realization $|X|$
obtained by taking geometric $k$-simplices for each element of $X^{(k)}$ and then gluing
these simplices together using the boundary maps. Whenever we talk about topological properties
of a semisimplicial set, we are referring to its geometric realization.  An action of a
group $G$ on a semisimplicial set $X$ consists of actions of $G$ on each $X^{(k)}$ that commute
with the boundary maps.  This induces an action of $G$ on $|X|$.  The quotient
$X/G$ is naturally a semisimplicial set with $k$-simplices $X^{(k)}/G$.

\p{Definition of equivariant homology}
Let $G$ be a group and let $X$ be a semisimplicial set on which $G$ acts.  For a ring $R$, there are two
equivalent definitions of the equivariant homology groups $\HH_{\ast}^G(X;R)$:
\setlength{\parskip}{0pt}
\begin{compactitem}
\item Let $EG$ be a contractible semisimplical set on which $G$ acts freely, so
$EG/G$ is a $K(G,1)$.  The group $G$ then acts freely on $EG \times X$, and
$\HH_{\ast}^G(X;R)$ is the homology with coefficients in $R$ of the quotient
space $(EG \times X)/G$.  
\item Let $F_{\bullet} \rightarrow \Z$ be a projective resolution of the trivial $\Z[G]$-module $\Z$ and
let $C_{\bullet}(X;R)$ be the simplicial chain complex of $X$ with coefficients in $R$.  Then
$\HH_{\ast}^G(X;R)$ is the homology of the double complex $F_{\bullet} \otimes C_{\bullet}(X;R)$.
\end{compactitem}
Neither of these definitions depends on any choices. \setlength{\parskip}{\baselineskip}

\p{Functoriality}
Equivariant homology is functorial in the following sense.
If $G$ and $G'$ are groups acting on semisimplicial sets $X$ and $X'$, respectively,
and if $f\colon G \rightarrow G'$ is a group homomorphism and $\phi\colon X \rightarrow X'$
is a map such that $\phi(gx) = f(g) \phi(x)$ for all $g \in G$ and $x \in X$, then
we get an induced map $\HH_{\ast}^G(X;R) \rightarrow \HH_{\ast}^{G'}(X';R)$.

\p{Map to a point}
If $\{p_0\}$ is a single point on which $G$ acts trivially, then
$\HH_{\ast}^{G}(\{p_0\};R) = \HH_{\ast}(G;R)$.  For an arbitrary
semisimplicial set $X$ on which $G$ acts, the projection $X \rightarrow \{p_0\}$
thus induces a map $\HH_{\ast}^G(X;R) \rightarrow \HH_{\ast}(G;R)$.  For
this map, we have the following lemma.

\begin{lemma}
\label{lemma:highconnected}
Let $X$ be an $n$-connected semisimplicial set on which a group $G$ acts and
let $R$ be a ring.  Then then natural map $\HH_k^G(X;R) \rightarrow \HH_k(G;R)$
is an isomorphism for $k \leq n$ and a surjection for $k=n+1$.
\end{lemma}
\begin{proof}
See \cite[Proposition VII.7.3]{BrownCohomology}.  This
reference assumes that $X$ is contractible, but its proof gives the desired conclusion
when $X$ is assumed to be merely $n$-connected.
\end{proof}

\p{The spectral sequence}
One of the main calculational tools for equivariant homology is as follows.

\begin{lemma}
\label{lemma:equivspectralsequence}
Let $X$ be a semisimplicial set on which a group $G$ acts and let $R$ be a ring.
For each simplex $\sigma$ of $X/G$, let $\tsigma$ be a lift of $\sigma$ to $X$ and
let $G_{\tsigma}$ be the stabilizer of $\tsigma$.  Then there is a spectral
sequence
\[E^1_{pq} \cong \bigoplus_{\sigma \in (X/G)^{(q)}} \HH_p(G_{\tsigma};R) \Rightarrow \HH_{p+q}^G(X;R).\]
\end{lemma}
\begin{proof}
See \cite[VII.(7.7)]{BrownCohomology}.
\end{proof}

\p{Group actions on equivariant homology}
Now let $\Gamma$ be a group acting on a semisimplicial set $X$ and let $G$ be a normal subgroup of $\Gamma$.
For $\gamma \in \Gamma$, the maps $G \rightarrow G$ and $X \rightarrow X$ taking $g \in G$ to $\gamma g \gamma^{-1}$
and $x \in X$ to $\gamma x$ induce a map $\HH_{\ast}^{G}(X;R) \rightarrow \HH_{\ast}^G(X;R)$.  This recipe
gives an action of $\Gamma$ on $\HH_{\ast}^{G}(X;R)$.  The restriction of this action to $G$ is trivial (this
can be proved just like \cite[Proposition III.8.1]{BrownCohomology}, which proves that inner automorphisms
act trivially on ordinary group homology), so we get an induced action
of $\Gamma/G$ on $\HH_{\ast}^G(X;R)$.  It is clear from its
construction that the spectral sequence in Lemma
\ref{lemma:equivspectralsequence} is a spectral sequence of $R[\Gamma/G]$-modules.

\subsection{The complex of lines}
\label{section:lines}

Let $\cO$ be the ring of integers in a number field $K$ and let $P$ be a finite-rank projective
$\cO$-module.  Assume that $\cO^{\times}$ has an element of norm $-1$, so we can talk about
the group $\CL(P)$.  This group acts on the following space.

\begin{definition}
Let $\cO$ be the ring of integers in a number field $K$ and let $P$ be a finite-rank projective $\cO$-module.
A {\em line decomposition} of $P$ is an ordered sequence $(L_1,\ldots,L_n)$ of rank-$1$ projective
submodules of $P$ such that $P = L_1 \oplus \cdots \oplus L_n$.  The {\em complex of lines} in $P$, denoted
$\Lines(P)$, is the semisimplicial set whose $(k-1)$-simplices are ordered sequences $(L_1,\ldots,L_k)$
of rank-$1$ projective submodules of $P$ that can be extended to a line decomposition $(L_1,\ldots,L_n)$.
\end{definition}

We thus have the equivariant homology groups $\HH_k^{\CL(P)}(\Lines(P);\Q)$.  Our main result
about these equivariant homology groups is as follows.

\begin{lemma}
\label{lemma:surjection}
Let $\cO$ be the ring of integers in a number field $K$ and let $P$ be a finite-rank projective $\cO$-module.
Assume that $\cO^{\times}$ has an element of norm $-1$.  Then the natural
map $\HH_k^{\CL(P)}(\Lines(P);\Q) \rightarrow \HH_k(\CL(P);\Q)$ is a surjection for
$0 \leq k \leq \rk(P)-1$.
\end{lemma}
\begin{proof}
If $\Lines(P)$ were $(\rk(P)-2)$-connected, then this would follow from Lemma \ref{lemma:highconnected}.  Unfortunately,
this is not known and is likely to be false\footnote{In \cite[Proposition 2.12]{MillerPatztWilsonYasaki}, it
is proven that $\Lines(\cO^2)$ is not connected when $\cO$ is quadratic imaginary but not a PID; however,
such $\cO$ cannot contain an element of norm $-1$, so this does not quite give a counterexample
under our given hypotheses.} -- a slight strengthening of this would allow one to run the
argument used to prove \cite[Theorem A]{ChurchFarbPutman} and prove a result contradicting \cite[Theorem B$'$ in \S 5.3]{ChurchFarbPutman}.
We will need an alternative approach.

Let $n = \rk(P)$.  By the classification of finitely generated projective modules over Dedekind domains (see,
e.g.\ \cite[\S 1]{MilnorKTheory}), there exist a nonzero ideal $I \subset \cO$ such
that $P = \cO^{n-1} \oplus I$.  Using this identification, we see that
every element of $\GL(P)$ is of the form
\[\left(
\begin{array}{@{}c|c@{}}
A & B \\
\hline
C & D
\end{array}\right)
\quad \text{with $A \in \Mat_{n-1,n-1}(\cO)$, $B \in \Mat_{n-1,1}(I^{-1})$, $C \in \Mat_{1,n-1}(I)$, $D \in \cO$}.\]
Here $I^{-1} \subset K$ is the inverse of $I$ using the usual multiplication of 
fractional ideals in a Dedekind domain.  Define $\Gamma' = \GL(P) \cap \CL_n(\cO)$, so
$\Gamma'$ is finite-index in both $\CL(P)$ and $\CL_n(\cO)$.  Modulo $I$, the last
row of an element of $\Gamma'$ is of the form $(0,\ldots,0,\ast)$.  Define
$\Gamma$ to be the subgroup of matrices in $\Gamma'$ whose last row equals
$(0,\ldots,0,1)$ modulo $I$.  The group $\Gamma$ is thus finite-index in
$\Gamma'$.

We now construct a space for $\Gamma$ to act on.
Define $B_n(\cO,I)$ to be the semisimplicial set whose $(m-1)$-simplices 
are ordered sequences $(v_1,\ldots,v_m)$ of elements of $\cO^n$ that can be 
extended to a sequence $(v_1,\ldots,v_n)$ with the following properties:
\setlength{\parskip}{0pt}
\begin{compactitem}
\item The $v_i$ form a free $\cO$-basis for $\cO^n$.
\item The last coordinate of each $v_i$ equals either $0$ or $1$ modulo $I$.
\item Precisely $1$ of the $v_i$ has a last coordinate equal to $1$ modulo $I$.
\end{compactitem}
The action of the group $\Gamma$ on $\cO^n$ fixes the last coordinate modulo $I$.  It
follows that $\Gamma$ acts on $B_n(\cO,I)$.  We will prove in Lemma \ref{lemma:connect}
below that $B_n(\cO,I)$ is $(n-2)$-connected; in fact, this result was {\em almost}
proved in \cite{ChurchFarbPutman}, and we will show how to derive it from results in
this paper. For now we will continue with the proof of Lemma
\ref{lemma:surjection} assuming that $B_n(\cO,I)$ is $(n-2)$-connected.\setlength{\parskip}{\baselineskip}

We now come to the key fact that relates the above to $\Lines(P)$ and $\CL(P)$:

\begin{claim}
There is a simplicial map $\Psi\colon B_n(\cO,I) \rightarrow \Lines(P)$ taking
a vertex $v$ of $B_n(\cO,I)$ to the $\cO$-submodule $P \cap (\cO \cdot v)$ of
$P$.
\end{claim}
\begin{proof}[Proof of claim]
It is enough to prove that if $(v_1,\ldots,v_n)$ is a top-dimensional simplex
of $B_n(\cO,I)$, then $(P \cap (\cO \cdot v_1),\ldots,P \cap (\cO \cdot v_n))$
is a line decomposition of $P$.  In other words, letting
$L_i = P \cap (\cO \cdot v_i)$ we must prove that 
$P = L_1 \oplus \cdots \oplus L_n$.

Consider $x \in P$.  We must prove that $x$ can be uniquely expressed
as $x = x_1 + \cdots + x_n$ with $x_i \in L_i$ for $1 \leq i \leq n$.
Since the $v_i$ form a free $\cO$-basis of $\cO^n$ and $P \subset \cO^n$, there
exists unique $\lambda_1,\ldots,\lambda_n \in \cO$ such that
$x = \lambda_1 v_1 + \cdots + \lambda_n v_n$.  We have to show that
$\lambda_i v_i \in P$ for all $1 \leq i \leq n$.

Let $1 \leq i_0 \leq n$ be the unique index such that the last coordinate
of $v_{i_0}$ equals $1$ modulo $I$.  For $1 \leq i \leq n$ with $i \neq i_0$,
the last coordinate of $v_i$ thus equals $0$ modulo $I$, so $v_i \in P$ and thus
$\lambda_i v_i \in P$.  As for $\lambda_{i_0} v_{i_0}$, we have
\[\lambda_{i_0} v_{i_0} = x - \sum_{i \neq i_i} \lambda_i v_i.\]
Each term on the right hand side is an element of $P$, so $\lambda_{i_0} v_{i_0}$
is as well.
\end{proof}

The map $\Psi$ along with the inclusion $\Gamma \hookrightarrow \CL(P)$ induces a map
$\HH_k^{\Gamma}(B_n(\cO,I);\Q) \rightarrow \HH_k^{\CL(P)}(\Lines(P);\Q)$.  This map
fits into a commutative diagram
\[\begin{CD}
\HH_k^{\Gamma}(B_n(\cO,I);\Q) @>>> \HH_k^{\CL(P)}(\Lines(P);\Q) \\
@VVV                                    @VVV \\
\HH_k(\Gamma;\Q)                   @>>> \HH_k(\CL(P);\Q).
\end{CD}\] Since $\Gamma$ is a finite-index subgroup of $\CL(P)$, the
existence of the transfer map (see \cite[Proposition
III.10.4]{BrownCohomology}) implies that the bottom row of this
diagram is a surjection.  Since $B_n(\cO,I)$ is $(n-2)$-connected by
Lemma \ref{lemma:connect} below, Lemma \ref{lemma:highconnected}
implies that the left column of this diagram is a surjection.  We
conclude that the right column of this diagram is a surjection, as
desired.
\end{proof}

It remains to prove the following result, which was promised during the above
proof.

\begin{lemma}
\label{lemma:connect}
Let $\cO$ be the ring of integers in a number field $K$ and let $I \subset \cO$
be a nonzero ideal.  Assume that $\cO$ has a real embedding (which hold, for instance,
if $\cO^{\times}$ has an element of norm $-1$).
Then the space $B_n(\cO,I)$ defined in the
proof of Lemma \ref{lemma:surjection} above is $(n-2)$-connected.
\end{lemma}
\begin{proof}
We start by introducing an auxiliary space.  Define $\hB_n(\cO,I)$ to be
the simplicial complex whose $(m-1)$-simplices are unordered sets
$\{v_1,\ldots,v_m\}$ of elements of $\cO^n$ such that some
(equivalently, any) ordering is an $(m-1)$-simplex of $B_n(\cO,I)$.

Recall that a simplicial complex $X$ is said to be {\em weakly 
Cohen–-Macaulay of dimension $r$} if it satisfies the following two properties:
\setlength{\parskip}{0pt}
\begin{compactitem}
\item $X$ is $(r-1)$-connected.
\item For all $m$-dimensional simplices $\sigma$ of $X$, the link
$\link_X(\sigma)$ of $\sigma$ in $X$ is $(r-m-2)$-connected.
\end{compactitem}
These two conditions can be combined if you regard $\sigma = \emptyset$ as a $-1$-simplex
of $X$ with $\link_X(\sigma) = X$. \setlength{\parskip}{\baselineskip}

By definition, the only difference between $B_n(\cO,I)$ and $\hB_{n}(\cO,I)$ is
that the vertices in a simplex of $B_n(\cO,I)$ are ordered.  In
\cite[Proposition 2.14]{RandalWilliamsWahl}, it is proved that in this
situation, if $\hB_{n}(\cO,I)$ is weakly Cohen--Macaulay of dimension $(n-1)$, then
$B_n(\cO,I)$ is $(n-2)$-connected.  To prove the lemma, therefore, it is enough
to prove that $\hB_n(\cO,I)$ is weakly Cohen--Macaulay of dimension $(n-1)$.

We now introduce yet another space.  Define $\hB'_n(\cO,I)$ to be the
simplicial complex whose $(m-1)$-simplices are unordered sets $\{v_1,\ldots,v_m\}$
of elements of $\cO^n$ that can be
extended to an unordered set $\{v_1,\ldots,v_n\}$ with the following properties:
\setlength{\parskip}{0pt}
\begin{compactitem}
\item The $v_i$ form a free $\cO$-basis for $\cO^n$.
\item The last coordinate of each $v_i$ equals either $0$ or $1$ modulo $I$.
\end{compactitem}
We thus have $\hB_n(\cO,I) \subset \hB'_n(\cO,I)$.  In 
\cite[Theorem E$'$ from \S 2.3]{ChurchFarbPutman}, Church--Farb--Putman
proved that $\hB'_n(\cO,I)$ is weakly Cohen--Macaulay of dimension $(n-1)$.  This
is where we use the assumption that $\cO$ has a real embedding. \setlength{\parskip}{\baselineskip}

We now show how to use the fact that $\hB'(\cO,I)$ is weakly Cohen--Macaulay of dimension $(n-1)$
to prove the same fact for $\hB(\cO,I)$.
Let $\sigma$ be an $m$-simplex of $\hB_n(\cO,I)$, where we allow $\sigma = \emptyset$ and $m=-1$.
We then have
\[\link_{\hB(\cO,I)}(\sigma) \subset \link_{\hB'(\cO,I)}(\sigma).\]
Since $\hB'_n(\cO,I)$ is weakly Cohen-Macaulay of dimension $(n-1)$, the space
$\link_{\hB'(\cO,I)}(\sigma)$ is $(n-m-3)$-connected.  To prove the same for
$\link_{\hB(\cO,I)}(\sigma)$, it is enough to construct a retraction  
$\rho\colon \link_{\hB'(\cO,I)}(\sigma) \rightarrow \link_{\hB(\cO,I)}(\sigma)$.
There are two cases.

\begin{casea}
There exists a vertex $w$ of $\sigma$ whose last coordinate equals $1$ modulo $I$.
\end{casea}
\begin{proof}[Proof of case]
In this case, the complex $\link_{\hB(\cO,I)}(\sigma)$ is the full subcomplex of $\link_{\hB'(\cO,I)}(\sigma)$
spanned by vertices whose last coordinates equal $0$ modulo $I$.
For all vertices $v$ of $\link_{\hB'(\cO,I)}(\sigma)$, we define
\[\rho(v) = \begin{cases}
v-w & \text{if the last coordinate of $v$ equals $1$ modulo $I$},\\
v   & \text{otherwise}.
\end{cases}\]
The last coordinate of $\rho(v)$ thus equals $0$ modulo $I$.  This extends to a simplicial
retraction $\rho\colon \link_{\hB'(\cO,I)}(\sigma) \rightarrow \link_{\hB(\cO,I)}(\sigma)$
due to the following fact:
\setlength{\parskip}{0pt}
\begin{compactitem}
\item If $\{v_1,\ldots,v_n\}$ is a free $\cO$-basis of $\cO^n$ and $c_2,\ldots,c_n \in \cO$, then
$\{v_1,v_2+c_2 v_1,\ldots,v_n+c_n v_1\}$ is a free $\cO$-basis of $\cO^n$.
\end{compactitem}
This completes the proof for this case.\setlength{\parskip}{\baselineskip}
\end{proof}

\begin{casea}
The last coordinate of all vertices of $\sigma$ equals $0$ modulo $I$.
\end{casea}
\begin{proof}[Proof of case]
In this case, the complex $\link_{\hB(\cO,I)}(\sigma)$ is the subcomplex of $\link_{\hB'(\cO,I)}(\sigma)$
consisting of simplices that contain no edges between vertices both of whose last coordinates equal $1$ modulo $I$.
We remark that this is not a full subcomplex.

Let $\cE$ be the set of edges of $\link_{\hB'(\cO,I)}(\sigma)$ joining vertices both of whose last coordinates equal
$1$ modulo $I$.  The retraction we will construct will depend on two arbitrary choices:
\setlength{\parskip}{0pt}
\begin{compactitem}
\item An enumeration $\cE = \{e_1,e_2,\ldots\}$.
\item For each $i \geq 1$, an enumeration $e_i = \{w_i,w_i'\}$ of the two vertices of $e_i$.  Since $w_i$ and $w'_i$ are distinct
vertices of $\link_{\hB'(\cO,I)}(\sigma)$ whose last coordinates are $1$ modulo $I$, we have that $w_i-w'_i$ is a vertex
of $\link_{\hB'(\cO,I)}(\sigma)$ whose last coordinate is $0$ modulo $I$.
\end{compactitem}
For each $i \geq 1$, we define a map $\rho_i\colon \link_{\hB'(\cO,I)}(\sigma) \rightarrow \link_{\hB'(\cO,I)}(\sigma)$
as follows.  First, let $S_i$ be the result of subdividing the edge $e_i$ of $\link_{\hB'(\cO,I)}(\sigma)$ with a new vertex $\kappa_i$.   
We then define a simplicial map $\rho'_i\colon S_i \rightarrow \link_{\hB'(\cO,I)}(\sigma)$ via the formula \setlength{\parskip}{\baselineskip}
\[\rho'_i(v) = \begin{cases}
w_i-w'_i & \text{if $v = \kappa_i$},\\
v        & \text{if $v \neq \kappa_i$}.
\end{cases} \quad \quad (\text{$v$ a vertex of $S_i$}).\]
This is a simplicial map for the following reason.  It is clear that $\rho'_i$ extends over the simplices of $S_i$ that do
not contain $\kappa_i$.  The simplices of $S_i$ that do contain $\kappa_i$ are of the form
\[\{w_i,\kappa_i,v_3,\ldots,v_m\} \quad \text{and} \quad \{\kappa_i,w'_i,v_3,\ldots,v_m\}\]
for a simplex $\{w_i,w'_i,v_3,\ldots,v_m\}$ of $\link_{\hB'(\cO,I)}(\sigma)$.  The images under $\rho'_i$ of these
two simplices are
\[\{w_i,w_i-w'_i,v_3,\ldots,v_m\} \quad \text{and} \quad \{w_i-w'_i,w'_i,v_3,\ldots,v_m\},\]
both of which are simplices of $\link_{\hB'(\cO,I)}(\sigma)$.  The map $\rho_i$ is then the composition
of $\rho'_i$ with the (nonsimplicial) subdivision map $\link_{\hB'(\cO,I)}(\sigma) \stackrel{\cong}{\rightarrow} S_i$.

Now define $\rho\colon \link_{\hB'(\cO,I)}(\sigma) \rightarrow \link_{\hB'(\cO,I)}(\sigma)$ to be the composition
\[\link_{\hB'(\cO,I)}(\sigma) \stackrel{\rho_1}{\longrightarrow} \link_{\hB'(\cO,I)}(\sigma) \stackrel{\rho_2}{\longrightarrow} \link_{\hB'(\cO,I)}(\sigma) \stackrel{\rho_3}{\longrightarrow} \cdots.\]
This infinite composition makes sense and is continuous since for each simplex $\sigma$ of $\link_{\hB'(\cO,I)}(\sigma)$, the sequence
\[\sigma, \rho_1(\sigma), \rho_2 \circ \rho_1(\sigma), \rho_3 \circ \rho_2 \circ \rho_1(\sigma),\ldots\]
of subsets eventually stabilizes. These images are not simplices, but rather finite unions of simplices.  From
its construction, it is clear that $\rho$ is a retraction from $\link_{\hB'(\cO,I)}(\sigma)$ to $\link_{\hB(\cO,I)}(\sigma)$.
\end{proof}

This completes the proof of the lemma.
\end{proof}
\subsection{The proof of Proposition \ref{proposition:projectiveactiontrivial}}
\label{section:projectiveactiontrivialproof}

We finally prove Proposition
\ref{proposition:projectiveactiontrivial}, which completes the proof
Theorem \ref{maintheorem:highsmall}.

\begin{proof}[Proof of Proposition \ref{proposition:projectiveactiontrivial}]
We start by recalling the setup.  
Let $\cO$ be the ring of integers in a number field $K$ and let
$P$ be a finite-rank projective $\cO$-module.
Assume that $\cO^{\times}$ has an element of norm $-1$, and
let $r$ and $2s$ be the numbers of real and complex embeddings of $K$.
We must prove that the action of $\GL(P)$ on its normal subgroup $\CL(P)$ induces
the trivial action on $\HH_k(\CL(P);\Q)$ for
$0 \leq k \leq \min(r+s,\rk(P))-1$.  This action factors through $\GL(P)/\CL(P) \cong \Z/2$.

The group $\GL(P)$ acts on both $\CL(P)$ and on the complex of lines
$\Lines(P)$.  We thus get an induced action of $\GL(P)/\CL(P)$ on
$\HH_k^{\CL(P)}(\Lines(P);\Q)$.  The natural map
$\HH_k^{\CL(P)}(\Lines(P);\Q) \rightarrow \HH_k(\CL(P);\Q)$ is
$\GL(P)/\CL(P)$-equivariant, and by Lemma \ref{lemma:surjection} is
also surjective for $0 \leq k \leq \rk(P)-1$.  We deduce that to prove
that the action of $\GL(P)/\CL(P)$ on $\HH_k(\CL(P);\Q)$ is trivial
for $0 \leq k \leq \min(r+s,\rk(P))-1$, it is enough to prove that the
$\GL(P)/\CL(P)$-action on $\HH_k^{\CL(P)}(\Lines(P);\Q)$ is trivial
for $0 \leq k \leq r+s-1$.

By Lemma \ref{lemma:equivspectralsequence} (and the paragraph following that lemma), 
we have a spectral sequence of $\Q[\GL(P)/\CL(P)]$-modules of the form
\begin{equation}
\label{eqn:thespec}
E^1_{pq} = \bigoplus_{\sigma \in \Lines(P)^{(q)}/\CL(P)} \HH_p((\CL(P))_{\tsigma};\Q) \Rightarrow \HH_{p+q}^{\CL(P)}(\Lines(P);\Q).
\end{equation}
Here $\tsigma \in \Lines(P)^{(k)}$ is an arbitrary lift of $\sigma$.  The key to the proof is the following.

\begin{claim}
The group $\GL(P)/\CL(P)$ acts trivially on $E^1_{pq}$ for $0 \leq p \leq r+s-1$.
\end{claim}
\begin{proof}[Proof of claim]
The group $\GL(P)/\CL(P)$ acts trivially on the set $\Lines(P)^{(q)}/\CL(P)$, so it does not permute
the terms in \eqref{eqn:thespec}.  Let $\tsigma = (L_1,\ldots,L_{q-1})$ be a $q$-simplex of $\Lines(P)$ such
that $\HH_p((\CL(P))_{\tsigma};\Q)$ is one of the terms in \eqref{eqn:thespec}.  We must prove
that the group $\GL(P)/\CL(P)$ acts trivially on $\HH_p((\CL(P))_{\tsigma};\Q)$.
Since $\GL(P)/\CL(P) \cong \Z/2$, it is enough to find a single element of $\GL(P) \setminus \CL(P)$ that
acts trivially.  Extend $\tsigma$ to a line decomposition $(L_1,\ldots,L_{\rk(P)})$ of $P$.  Set
$P' = L_1 \oplus \cdots \oplus L_{q-1}$ and $P'' = L_q \oplus \cdots \oplus L_{\rk(P)}$, so
$P = P' \oplus P''$.  Let $x$ be an element of $\GL(L_1)$ that does not lie in $\CL(L_1)$ and let
\[X = (x,1,\ldots,1) \in \GL(L_1) \times \GL(L_2) \times \cdots \times \GL(L_{q-1}) \times \GL(P'') \subset \GL(P).\]
Since $\GL(L_i) \cong \cO^{\times}$ is abelian, the element $X$ commutes with the subgroup
\[\Lambda = \CL(L_1) \times \CL(L_2) \times \cdots \times \CL(L_{q-1}) \times \CL(P'')\]
of $\CL(P)$.  It follows that $X$ acts trivially on the image of $\HH_p(\Lambda;\Q)$ in
$\HH_p((\CL(P))_{\sigma};\Q)$.  It is enough, therefore, to prove that the map
$\HH_p(\Lambda;\Q) \rightarrow \HH_p((\CL(P))_{\tsigma};\Q)$ is surjective.

It follows from Lemma \ref{lemma:decompose} that the $\GL(P)$-stabilizer
of $\tsigma$ can be written as
\[\Hom(P'',P') \rtimes \left(\GL\left(L_1\right) \times \cdots \times \GL\left(L_{q-1}\right) \times \GL\left(P''\right)\right).\]
From this, we see that
\[\Lambda' = \Hom(P'',P') \rtimes \left(\CL\left(L_1\right) \times \cdots \times \CL\left(L_{q-1}\right) \times \CL\left(P''\right)\right)\]
is a finite-index subgroup of $\CL(P)_{\tsigma}$.  The existence of the transfer map
(see \cite[Proposition III.10.4]{BrownCohomology}) implies that the map $\HH_p(\Lambda';\Q) \rightarrow \HH_p(\CL(P)_{\tsigma};\Q)$ is surjective.  Finally, Lemma \ref{lemma:spectralsequence} (with $G = \CL(L_1) \times \cdots \times \CL(L_{q-1})$)
implies that the map $\HH_p(\Lambda;\Q) \rightarrow \HH_p(\Lambda';\Q)$ is surjective (this is where
we use the assumption that $0 \leq p \leq r+s-1$).  We conclude that the map
$\HH_p(\Lambda;\Q) \rightarrow \HH_p((\CL(P))_{\tsigma};\Q)$ is surjective, as desired.
\end{proof}

Now, the spectral sequence \eqref{eqn:thespec} computes the associated graded of a 
filtration $\cF_{\bullet}$ of $\Q[\GL(P)/\CL(P)]$-modules on 
$\HH_{k}^{\CL(P)}(\Lines(P);\Q)$ for
each $k$.  The above claim implies that $\GL(P)/\CL(P)$ acts trivially
on $E^{\infty}_{pq}$ for $0 \leq p \leq r+s-1$, so for $0 \leq k \leq r-s-1$ 
the $\GL(P)/\CL(P)$-action on the 
associated graded terms for the filtration 
$\cF_{\bullet} \HH_{k}^{\CL(P)}(\Lines(P);\Q)$ are trivial.
Since $\GL(P)/\CL(P) = \Z/2$ is a finite
group, Maschke's theorem implies that the category of $\Q[\GL(P)/\CL(P)]$-modules is semisimple, so this
implies that the $\GL(P)/\CL(P)$-action on $\HH_{k}^{\CL(P)}(\Lines(P);\Q)$ for $0 \leq k \leq r+s-1$ is also trivial.  The
lemma follows.
\end{proof}

\begin{footnotesize}
\noindent
\begin{tabular*}{\linewidth}[t]{@{}p{\widthof{Department of Mathematics}+0.5in}@{}p{\linewidth - \widthof{Department of Mathematics} - 0.5in}@{}}
{\raggedright
Andrew Putman\\
Department of Mathematics\\
University of Notre Dame \\
255 Hurley Hall\\
Notre Dame, IN 46556\\
{\tt andyp@nd.edu}}
&
{\raggedright
Daniel Studenmund\\
Department of Mathematics\\
University of Notre Dame \\
255 Hurley Hall\\
Notre Dame, IN 46556\\
{\tt dstudenm@nd.edu}}
\end{tabular*}\hfill
\end{footnotesize}

\end{document}